\documentclass[11pt]{amsart}

\usepackage[a4paper,margin=25mm]{geometry}
\usepackage{amsmath,amssymb,amsthm,mathtools}
\usepackage{tikz}
\usetikzlibrary{positioning}
\usepackage[colorlinks=true,linkcolor=blue,citecolor=blue,urlcolor=blue]{hyperref}

\usepackage[shortlabels]{enumitem} 

\newtheorem{theorem}{Theorem}[subsection]
\newtheorem{proposition}[theorem]{Proposition}
\newtheorem{lemma}[theorem]{Lemma}
\newtheorem{corollary}[theorem]{Corollary}
\newtheorem{definition}[theorem]{Definition}
\newtheorem{remark}[theorem]{Remark}
\newtheorem{example}[theorem]{Example}

\newcommand{\Aut}{\operatorname{Aut}}
\newcommand{\Out}{\operatorname{Out}}
\newcommand{\Inn}{\operatorname{Inn}}
\newcommand{\Stab}{\operatorname{Stab}}

\newcommand{\Conj}{\operatorname{Conj}}
\newcommand{\ad}{\operatorname{ad}}
\newcommand{\Z}{\mathbb Z}

\newcommand{\cG}{\mathcal G}

\newcommand{\ol}[1]{\overline{#1}}

\title[Conjugacy in cyclic extensions]
{The conjugacy problem in cyclic extensions\\
of one-ended hyperbolic groups}
\author{Armando Martino}
\date{\today}

\begin{document}

\begin{abstract}
Let \(G\) be a torsion-free one-ended hyperbolic group and let
\(\phi\in\Aut(G)\). We prove that the mapping torus, or \textit{suspension}, 
\[
M:=G_{\phi}=G\rtimes_\phi\Z
\]
has solvable conjugacy problem. This builds on the pioneering work of Pr\'eaux in \cite{PreauxOrientable} and \cite{PreauxNonorientable}, who solved the conjugacy problem for all (geometrisable) three-manifolds. We view \( M\) as a generalisation of a fibred three-manifold.

Our proof begins with the canonical JSJ tree of \(G\), whose
suspension gives a graph-of-groups decomposition of \( M\). We
refine each suspended QH vertex using a Nielsen--Thurston reduction
system for its induced monodromy, taking account of non-orientable
surfaces and orientation-reversing monodromy. Equivalently, this is a further geometric JSJ decomposition for each vertex which is a fibred three-manifold, but allowing Klein bottles as well as tori in the splitting.

We then `block' pieces of this refined JSJ together according to whether they share a central element, glued along an elementary vertex, by a sequence of folding operations. This new splitting is cocompact and acylindrical; in particular certain local `solution sets' turn out to be rational subsets of virtually abelian subgroups - the virtually abelian groups in question are products of the edge groups. Our version of Pr\'eaux's
graph-of-groups argument then reduces global conjugacy to effective
intersection and non-emptiness for these rational sets.
\end{abstract}

\date{}
\maketitle

\section{Introduction}
\subsection{Statement and outline}

The conjugacy problem in a finitely generated group asks for an algorithm which, given two words, decides whether the elements which
they represent are conjugate. The conjugacy problem is solvable in hyperbolic groups and in fundamental groups of compact
three-manifolds. It is not, however, inherited by arbitrary group extensions, nor does it in general pass from a subgroup of finite
index to the ambient group; see Collins--Miller \cite{CollinsMiller}.

The purpose of this paper is to prove the following.

\begin{theorem}\label{thm:main}
Let \(G\) be a torsion-free one-ended hyperbolic group and let \(\phi\in\Aut(G)\). Then the conjugacy problem in
\[
G_\phi=G\rtimes_\phi\langle t\rangle
\]
is solvable.
\end{theorem}

As usual, this is a statement for each fixed pair \((G,\phi)\); the
precise input convention is recorded in
Remark~\ref{rem:fixed-input}.

All group and tree actions in this paper are right actions. In particular, our conjugation convention is
\[
x^g=g^{-1}xg,\qquad H^g=g^{-1}Hg, \qquad \ad_g(x)=x^g.
\]
Thus \(G_\phi\) is defined by
\[
 G_{\phi}:=\langle G, t \ : \ g^t=\phi(g), \ \forall g \in G \rangle 
\]
We denote by
\[
\chi:G_\phi\longrightarrow\Z
\]
the height homomorphism, defined by \(\chi(G)=0\) and \(\chi(t)=1\).

There are two immediate cases. If the class of \(\phi\) has finite order in \(\Out(G)\), then \(G_\phi\) has infinite cyclic centre and
its central quotient is hyperbolic. The height homomorphism recovers the central coordinate, and the conjugacy problem in \(G_\phi\)
reduces directly to conjugacy in the hyperbolic quotient. If \(G\) is a closed surface group, then \(G_\phi\) is the fundamental group
of a fibred three-manifold, and the result follows from Pr\'eaux's solution of the conjugacy problem for geometrizable three-manifold
groups \cite{PreauxOrientable,PreauxNonorientable}.

We may therefore suppose that the image of \(\phi\) in \(\Out(G)\) has infinite order and that the canonical elementary JSJ splitting
of \(G\) is non-trivial. We use the canonical, bipartite tree of cylinders
description developed by Guirardel--Levitt \cite{GuirardelLevittCylinders,GuirardelLevittAutomorphisms,
GuirardelLevittJSJ}. One class of vertices has elementary stabiliser, and the other has either rigid or
quadratically hanging stabiliser. Since the tree is canonical, the action of \(G\) extends to an action of \(G_\phi\).

Following Dahmani--Krishna \cite{DahmaniKrishna}, we call
\(H\rtimes_\psi\Z\) the \emph{suspension of \(H\) by \(\psi\)}, or
simply a suspension. Thus the stabiliser in \(G_\phi\) of a vertex
is the suspension of a first-return automorphism of its stabiliser
in \(G\). This gives three initial kinds of vertex group:
\begin{enumerate}[(i)]
\item elementary suspensions, which are virtually \(\Z^2\) - either tori or Klein bottles;
\item rigid suspensions, whose centre is infinite cyclic and whose quotient by the centre is hyperbolic;
\item QH suspensions, which are fundamental groups of fibred three-manifolds.
\end{enumerate}
We call a suspension \(H\rtimes_\psi\Z\) \emph{finite-outer} when \( H\) is
non-elementary hyperbolic and its monodromy \( \psi \) has finite order in the
outer automorphism group of \( H\) - see Definition~\ref{def: finite outer}. Our point of view is to treat finite-outer suspensions as analogues of Seifert-fibred manifolds. 

We refine a QH suspension using the Nielsen--Thurston decomposition of its first-return monodromy. The canonical reduction system is
preserved by the original mapping class, although its curves and complementary components may be permuted and annular orientations may
be reversed. This is equally valid for a non-orientable QH surface, where the canonical system may contain one-sided curves
\cite{ParisNonorientable,Wu}. For each one-sided reduction curve, we split along the boundary of an equivariantly chosen M\"obius neighbourhood; Example~\ref{ex:mobius-extension} motivates this construction, while Lemma~\ref{lem:surface-tree} gives its precise form. Alternatively, one could split only along the two-sided curves of the reduction system. As explained in Remark~\ref{rem:root extensions}, the resulting pieces differ from the standard ones by root extensions, which are algebraically benign but spoil the geometry; the pieces are no longer hyperbolic or Seifert-fibred three-manifolds.

This gives a refinement for the full \(G_\phi\)-action without passing to a power. Dahmani--Krishna's
\(T_{\mathrm{pA}}\)-construction is a closely related precedent
\cite[\S2.4.2]{DahmaniKrishna}.

The new ingredient here is not the relative-hyperbolic decomposition
itself, but the block folding, Theorem~\ref{thm:block-reduction}, and the complete rational local solution
sets needed for a direct graph of groups conjugacy algorithm. The idea of the block folding is that the action of \( G_{\phi}\) on the refined JSJ tree need not be acylindrical. We therefore fold finite-outer suspension vertices until it becomes acylindrical. 

Dahmani and Krishna also show that \(G_\phi\) is relatively
hyperbolic with respect to the suspensions of its maximal
polynomially growing subgroups \cite[Theorem~1.2]{DahmaniKrishna}. 
Since the conjugacy problem in a relatively hyperbolic group reduces to the conjugacy
problems in its peripheral subgroups \cite[Theorem~1.1]{Bumagin}, one
could restrict the block construction below to these polynomial
suspensions and then invoke this reduction.

This saves little here. It removes only the pseudo-Anosov vertices from the final
graph-of-groups argument but their required local properties are
mostly supplied directly by Pr\'eaux, detailed in Proposition~\ref{prop:effective-preaux}, while the block analysis of every
polynomial component is unchanged. It would also introduce a
separate relative-hyperbolic reduction. We therefore retain the
pseudo-Anosov vertices and work directly with the entire refined JSJ.

Our proof is strongly modelled on Pr\'eaux's proof for three-manifold groups
\cite{PreauxOrientable}. Pr\'eaux first proves a conjugacy theorem for cyclically reduced forms in a graph of groups
\cite[Theorem~3.1]{PreauxOrientable}. He then reduces conjugacy in
the whole group to three problems in the vertex groups - see Definition~\ref{def: local problems}:
\begin{enumerate}[(i)]
\item the conjugacy problem;
\item the boundary-parallelism problem;
\item the two-coset problem.
\end{enumerate}
The reduction additionally uses effective cyclic reduction, controlled propagation of conjugacy through edge groups, and a precise description
of infinite families of two-coset solutions. We formulate the latter
as effective rationality in finite products of edge groups. See Definition~\ref{def:rational} and Proposition~\ref{prop:abstract-preaux} for details.

For the pseudo-Anosov surface-by-\(\Z\) vertices, the required local
algorithms are those established by Pr\'eaux, although we appeal to relative hyperbolic machinery to show the rationality of our local sets in Definition~\ref{def: local problems} in the non-orientable case. The other
non-elementary vertices are finite-outer suspensions. Each has an infinite cyclic central subgroup of non-zero height, hyperbolic
quotient, and virtually cyclic peripheral images. Their local problems therefore reduce to standard algorithms for virtually cyclic
subgroups of hyperbolic groups, together with rational constraints in
the central and edge-group coordinates.

Formally we rely on Pr\'eaux for our argument but in fact end up following his proof scheme and rework the relevant ingredients in our setting. In that sense, our result can be regarded as a generalisation of his. Moreover, the analogy between suspensions and fibred 3-manifolds is very strong and our block tree, constructed in section~\ref{sec:blocks}, is a close analogue of the JSJ decomposition for a 3-manifold and recovers it in the geometric setting.

The paper is organised as follows. In Section~\ref{sec:finite} we dispose of finite outer monodromy. In
Section~\ref{sec:preaux} we give a detailed account of the graph of groups machinery used by Pr\'eaux. The canonical
Guirardel--Levitt JSJ and its suspension are described in Section~\ref{sec:GL}. Section~\ref{sec:refinement} treats the
QH refinement and its three-manifold interpretation. Section~\ref{sec:blocks} proves the block-reduction theorem,
constructing a cocompact four-acylindrical tree \( T_{\mathrm{blk}}\) with uniform edge and
vertex types. Finally, Section~\ref{sec:local-verification} verifies
the hypotheses of the abstract Pr\'eaux reduction, one vertex type at
a time, and completes the proof.

Also, while not necessary for the main argument, we prove that \( T_{\mathrm{blk}}\) is invariant under \( \Aut(G_{\phi})\) in Proposition~\ref{prop:invariant block tree}. 
\subsection{Torsion and infinitely many ends}

There is an obvious question as to whether one can weaken the hypothesis on the hyperbolic group \( G\) to allow torsion. 
We do not expect this to be insurmountable but it would make many of the arguments more technical. Finite-outer suspensions still have (virtually) infinite cyclic centres, for instance. The real obstacle is probably the construction of the refined JSJ tree; this would require some analogue of the Nielsen-Thurston reduction system for 2-orbifolds. That said, it may be deducible from the surface case. In any case, it would be a somewhat lengthy extension which we do not attempt.

In earlier joint work with Bogopolski, Maslakova and Ventura, we proved
that free-by-cyclic groups have solvable conjugacy problem \cite{BogopolskiMartinoMaslakovaVentura}. The present theorem is
intended to provide the complementary one-ended, inductive input to a treatment
of cyclic extensions of arbitrary torsion-free hyperbolic groups, which we plan to give separately.

\begin{remark}[Fixed-pair input convention]\label{rem:fixed-input}
	For each fixed pair \((G,\phi)\), the input consists of two words in a fixed generating set of \(G_\phi\). Our algorithm is not uniform in the sense that we take our artefacts as part of the algorithm. That is, we fix once and for all
	\(T_{\mathrm{blk}}\) and its finite quotient graph of groups, together with the presentations, generators, edge maps, virtually abelian and
	hyperbolic quotients of the finite-outer vertices, peripheral conjugacy data, first-return and Nielsen--Thurston data, and translations between the original and
	Bass--Serre generators used below. Effective rational relations may be represented either by automata or by the equivalent coset-wise
	semilinear data of Theorem~\ref{thm:rational-calculus}.
	
	We do not claim a uniform procedure taking a presentation of \(G\) and words for \(\phi\) as input. To obtain one, one would need to
	construct the canonical JSJ and its induced \(\phi\)-action algorithmically, construct the tree \( T_{\mathrm{blk}}\), and extract all the finite data listed above; compare Barrett \cite{Barrett} for the first step. Once these data are available the remaining procedures are those proved or cited below. In
	particular, the partial quasiconvex-subgroup algorithms terminate because the relevant subgroups are known to be quasiconvex. We have tried to organise the material to make such a potential future push to uniformity more readily available.
\end{remark}

\subsection*{Acknowledgements}

I would like to thank Harry Iveson for reading and discussing
Pr\'eaux's paper \cite{PreauxOrientable} with me. I also thank Naomi Andrew and Gilbert Levitt for many conversations about JSJ splittings
and reduction systems beyond the orientable setting, and especially for drawing my attention to \cite{ParisNonorientable}. I'd also like to thank Sara Luder for pointing me in the direction of \cite{KharlampovichMiasnikovWeil}. 

\section{Finite outer monodromy}\label{sec:finite}
\subsection{The finite-outer case}

We start by dealing with the finite-outer case.

\begin{definition}[Finite-outer]
	\label{def: finite outer}
	We call a suspension \( G \rtimes_{\phi} \Z \) \textit{finite-outer} if \( G\) is non-elementary hyperbolic and \( \phi\) induces a finite order outer automorphism of \( G\). 
\end{definition}

Next we have a general observation about central extensions equipped with a height homomorphism.

\begin{lemma}[The height lemma]\label{lem:height}
Let \(V\) be a group, let \(C=\langle z\rangle\leq Z(V)\) be infinite
cyclic, and suppose that there is a homomorphism
\(\eta:V\to\Z\) such that \(\eta(z)\neq 0\). Write
\(\ol V=V/C\). For \(u,v\in V\),
\[
 u\sim_V v
 \quad\Longleftrightarrow\quad
 \eta(u)=\eta(v)\ \text{ and }\ \ol u\sim_{\ol V}\ol v.
\]
Consequently, if \(\ol V\) has solvable conjugacy problem, then so
does \(V\).
\end{lemma}

\begin{proof}
The forward implication is immediate. Conversely, suppose that \(\eta(u)=\eta(v)\) and that \(\ol u\) and \(\ol v\) are conjugate in
\(\ol V\). Lift a conjugator to an element \(h\in V\). There is then an integer \(k\) such that
\[
u^h=vz^k.
\]
Applying \(\eta\) gives
\[
\eta(u)=\eta(v)+k\eta(z).
\]
Since \(\eta(u)=\eta(v)\) and \(\eta(z)\neq 0\), we have \(k=0\).
\end{proof}

This allows us to deal with the finite-outer case easily. 

\begin{proposition}\label{prop:finite-outer}
Suppose that \(G\) is a non-elementary torsion-free hyperbolic group
and that the image of \(\phi\in\Aut(G)\) has finite order in
\(\Out(G)\). Then \(G_\phi:= G \rtimes_{\phi} \Z\) has infinite cyclic centre, the quotient
of \(G_\phi\) by its centre is hyperbolic, and \(G_\phi\) has solvable
conjugacy problem.
\end{proposition}

\begin{proof}
Choose \(m>0\) least such that \(\phi^m=\ad_a\) for some \(a\in G\); the hypotheses ensure this exists.

Since \(G\) is non-elementary and torsion-free hyperbolic, it has trivial centre. The two expressions
\(\phi\circ\phi^m\circ\phi^{-1}\) and \(\phi^m\) show that
\(\ad_{\phi(a)}=\ad_a\), so \(\phi(a)=a\). It follows that
\[
z=t^ma^{-1}
\]
commutes both with \(G\) and with \(t\). Thus \(z\) is central.

No non-trivial central element lies in \(G\), and hence the restriction of \(\chi\) to \(Z(G_\phi)\) is injective. If
\(c\in Z(G_\phi)\) has height \(n\), then conjugation by \(c\) on
\(G\) is trivial, while its outer automorphism class is
\([\phi]^n\). Thus \(\phi^n\) is inner, and the minimality of \(m\)
gives \(m\mid n\). Since \(z\) is central and \(\chi(z)=m\), it
follows that
\[
\chi\bigl(Z(G_\phi)\bigr)=m\Z.
\]
Injectivity of \(\chi\) on the centre now shows that \(z\) generates
\(Z(G_\phi)\).

The image of \(G\) in \(G_\phi/Z(G_\phi)\) therefore has finite
index \(m\). Hence \(G_\phi/Z(G_\phi)\) is hyperbolic, and the
conclusion follows from Lemma~\ref{lem:height}.
\end{proof}

\section{Pr\'eaux's graph of groups machinery}\label{sec:preaux}

We recall the part of \cite{PreauxOrientable} which will be used below. Although Pr\'eaux works with the graph of groups coming from
the JSJ decomposition of a Haken three-manifold, the conjugacy
criterion itself is the standard Bass--Serre conjugacy theorem.
We state it in the form most convenient for the present paper.

\subsection{Cyclically reduced forms}

Let \(\cG\) be a finite connected graph of groups with underlying
graph \(X\). For an oriented edge \(e\), write \(o(e)\) and \(t(e)\)
for its initial and terminal vertices, and
\[
 \iota_e^-:G_e\longrightarrow G_{o(e)},\qquad
 \iota_e^+:G_e\longrightarrow G_{t(e)}
\]
for the two edge monomorphisms. We identify \(G_{\bar e}=G_e\), with
\(\iota_{\bar e}^-=\iota_e^+\) and
\(\iota_{\bar e}^+=\iota_e^-\), and abbreviate
\(c^-=\iota_e^-(c)\) and \(c^+=\iota_e^+(c)\).

The Bass group of \(\cG\) is obtained from the free product of the
vertex groups by adjoining the oriented edges and imposing
\[
\bar e=e^{-1},
\qquad (c^-)^e=c^+\quad(c\in G_e).
\]
A graph-of-groups path from \(v_1\) to \(v_{n+1}\) is a word
\[
g_1e_1g_2e_2\cdots g_ne_ng_{n+1},
\qquad g_i\in G_{v_i},
\]
carried by an edge path
\[
(v_1,e_1,v_2,\ldots,e_n,v_{n+1}).
\]
For a chosen base vertex \(v\), the fundamental group
\(\pi_1(\cG,v)\) is the subgroup of the Bass group represented by
closed graph-of-groups paths based at \(v\). We suppress the
basepoint when it is irrelevant.

A subword \(eg\bar e\), with \(g\in G_{t(e)}\), is reducible precisely
when
\[
 g\in\iota_e^+(G_e).
\]
Writing \(g=c^+\), the defining relation gives the elementary
reduction
\[
e\,c^+\,\bar e\longmapsto c^-.
\]
The rule for \(\bar ege\) is the same rule applied to \(\bar e\). A path is \emph{reduced} if it admits no such elementary reduction,
and a closed path is \emph{cyclically reduced} if every cyclic
permutation is reduced. Its edge length is \(n\). Thus effective membership in incident edge-group images gives an algorithm for reducing and then cyclically reducing a word. These definitions and the normal-form theorem are standard; see \cite[I.5]{Serre} and \cite{DicksDunwoody}.

\begin{theorem}[Graph-of-groups conjugacy theorem]\label{thm:Collins}
Let \(u,u'\) be represented by cyclically reduced forms in
\(\pi_1(\cG)\).
\begin{enumerate}[(i)]
\item Conjugate elements have the same edge length.
\item Suppose both forms have length zero, with
  \(u\in G_v\) and \(u'\in G_{v'}\). Then \(u\) and \(u'\)
  are conjugate if and only if either they are conjugate in a
  common vertex group, or there is an edge path
  \[
   (v=v_0,e_1,v_1,\ldots,e_p,v_p=v')
  \]
  and elements \(c_i\in G_{e_i}\) such that
  \[
  u\sim_{G_{v_0}}c_1^-,
  \quad
  c_i^+\sim_{G_{v_i}}c_{i+1}^-\ (1\leq i<p),
  \quad
  c_p^+\sim_{G_{v_p}}u'.
  \]
\item Suppose both forms have positive length. After a cyclic permutation of one form, their underlying loops must agree:
  \[
  C=(v_1,e_1,\ldots,v_n,e_n).
  \]
  Writing the vertex labels as \(g_i,g_i'\), the elements are
  conjugate if and only if there exist \(c_i\in G_{e_i}\) such
  that
  \[
   g_1=(c_n^+)^{-1}g_1'c_1^-
  \]
  in \(G_{v_1}\), and
  \[
   g_i=(c_{i-1}^+)^{-1}g_i'c_i^-
   \qquad(2\leq i\leq n)
  \]
  in \(G_{v_i}\). In that case
  \[
  (u')^{c_n^+}=u
  \]
  in the whole group.
\end{enumerate}
\end{theorem}

This is Theorem~3.1 of Pr\'eaux \cite{PreauxOrientable}; compare Serre \cite[I.5]{Serre} and the conjugacy theorems for amalgamated
products and HNN extensions in \cite[Chapter~IV]{LyndonSchupp}. The last sentence of (iii) is the
graph-of-groups version of Collins' lemma. That is, for cyclically reduced words, conjugation involves cyclic permutation followed by conjugation by an element of an edge group.

\subsection{The three local problems}

Let \(V\) be a vertex group and let \(E,E'\leq V\) be images of
incident edge groups. We need to record not just the solvability of certain conjugacy related problems, but their entire solution sets so that we can enumerate all solutions and `intersect' them when needed to find consistent solutions. Compare with \cite[\S4]{PreauxOrientable}.


\begin{definition}
\label{def: local problems}	
	Let \( V\) be a group with subgroups, \( E, E' \leq V\). 

\noindent	
The \emph{boundary-parallelism problem} for \( u \in V \) is the set,
\[
\mathcal C_E(u)=\{e\in E\mid e\sim_V u\}.
\]	
The \emph{two-boundary conjugacy relation} is the set, 
\[
\mathcal P_{E,E'}(V)
=\{(e,e')\in E\times E'\mid e\sim_V e'\}.
\]	
Finally, the \emph{two-coset problem} for \( u , v \in V \) is the set
\[
\mathcal C_{E,E'}(u,v)
=\{(e,e')\in E\times E'\mid u=eve'\}.
\]	
\end{definition}

\begin{remark}
	The boundary-parallelism problem determines when an element fixes an edge in the associated tree. The two-boundary conjugacy relation allows one to determine if an element fixes edges from potentially different orbits. The two-coset problem is key in determining when two hyperbolic elements are conjugate. These are all used to solve the conjugacy problem in Proposition~\ref{prop:abstract-preaux} below. 
\end{remark}

For an algorithmic reduction it is not enough merely to decide whether these sets are empty. One needs an effective description
which is stable under the intersections, coordinate changes, and projections arising at the next vertices. The following language makes that requirement precise.

\begin{definition}[Effective rational edge data]\label{def:rational}
Let \(P\) be a finite product of edge groups, equipped with the fixed generating sets coming from the graph of groups. A subset \(R\subseteq
P\) is \emph{effectively rational} if an automaton accepting a regular
language whose image in \(P\) is \(R\) can be constructed. A rational subset of \(E\times E'\) will also be called a \emph{rational
relation} between \(E\) and \(E'\).
\end{definition}

We recall the standard facts that make this definition useful here. Rationality in a finitely generated group is independent of the
chosen finite generating set; see the discussion following \cite[Definition~3.6]{CiobanuEvetts}. Rational subsets are preserved
under homomorphic images, translations, finite unions, products, and star. In particular, rational relations can be treated simply as
rational subsets of direct products.

There is an important warning concerning intersection. Let
\(\pi:X^*\to G\) be the evaluation map, and suppose
\[
 R_i=\pi(L_i)
\]
for regular languages \(L_i\subseteq X^*\). The usual product
automaton constructs \(L_1\cap L_2\), but in general
\[
 \pi(L_1\cap L_2) \ \subsetneq\ \pi(L_1)\cap\pi(L_2);
\]
different words in the two languages may represent the same group
element. Note that we could replace \(L_i\) with \(\pi^{-1}(\pi(L_i))\), but this does not solve the problem as the larger set need not be regular. There is a difference between intersecting regular languages and intersecting rational subsets of a group; see \cite[\S3]{BartholdiSilva}.

This can happen in automatic groups. For example,
\[
 F(a,b)\times\Z
\]
is automatic. The subgroups
\[
 H=F(a,b)\times\{0\},
 \qquad
 K=\langle(a,0),(b,1)\rangle
\]
are finitely generated and hence rational, whereas
\[
 H\cap K =\{(w,0)\mid \operatorname{exp}_b(w)=0\}
\]
is a free group of infinite rank. A subgroup of a group is rational if and only if it is finitely generated
\cite[\S3.1]{BartholdiSilva}, so \(H\cap K\) is not rational.

Thus the only closure property used below which is not formal in an arbitrary group is intersection. It is available here because every
edge group is virtually \(\Z^2\), and hence every finite product of edge groups is virtually abelian. The virtually abelian hypothesis in the following theorem is therefore essential; it cannot be replaced
merely by automaticity.

A subset of \(\Z^d\) is \emph{semilinear} if it is a finite union of
sets of the form
\[
 a+\mathbb N v_1+\cdots+\mathbb N v_k, \qquad a,v_i\in\Z^d.
\]
The following theorem is not new. Its structural content is contained in the rational--semilinear calculus of Ciobanu--Evetts:
Theorem~3.8 and Lemma~3.9 of \cite{CiobanuEvetts}, together with their Propositions~3.4, 3.11 and~3.19 and Lemmas~3.18 and~3.20. We assemble
these results in the effective form used below and record explicitly how the relevant group operations appear in coset coordinates.

\begin{theorem}[Effective rational calculus in virtually abelian groups]
\label{thm:rational-calculus}
Let \(P\) be a finitely generated virtually abelian group, equipped
with a normal finite-index free-abelian subgroup
\(P_0\cong\Z^d\), a finite transversal, and the corresponding
coordinates.
\begin{enumerate}[label=\textup{(\roman*)}]
\item A subset \(R\subseteq P\) is rational if and only if its coordinates in each coset cell \(P_0t\) form a semilinear
  subset of \(\Z^d\). Conversion between automata and finite coset-wise semilinear descriptions is effective.
\item For finite products of such groups, effective rational subsets
  are effectively closed under finite unions and intersections, Cartesian products, coordinate projections, translations, and
  images and inverse images of maps which are integral affine on each of finitely many coset cells.
\item Multiplication, inversion, conjugation by a fixed element, and
  fixed homomorphisms between finitely generated virtually abelian groups are maps of the kind occurring in
  \textup{(ii)}, after a finite refinement of the coset cells.
\item Non-emptiness of an effective rational subset is decidable.
\end{enumerate}
In every part, the output may be given either by an automaton or by a
finite coset-wise semilinear description.
\end{theorem}

\begin{proof}
By \cite[Proposition~3.19]{CiobanuEvetts}, rational subsets of \(P\) are exactly its coset-wise polyhedral subsets. Proposition~3.11 and
Theorem~3.8(1) of that paper identify polyhedral, semilinear and rational subsets of the free-abelian cells. This gives the equivalence in \textup{(i)}.

The passage is effective with the fixed coordinate data. Indeed, \cite[Lemma~3.18]{CiobanuEvetts} gives the computable reduction of a
rational subset in a coset to one in \(P_0\), and the proof of \cite[Proposition~3.19]{CiobanuEvetts} gives the converse construction.
Within \(P_0\), the constructive Eilenberg--Sch\"utzenberger theorem
\cite{EilenbergSchutzenberger} converts between automata and semilinear descriptions. Appending or deleting the appropriate
transversal letter performs the conversion cell by cell.

Finite products are covered by \cite[Lemma~3.20]{CiobanuEvetts}. On each free-abelian cell,
intersection is Theorem~3.8(2) of that paper, integral-affine images are covered by Lemma~3.9, and integral-affine inverse images follow
from Propositions~3.4 and~3.11. Finite unions and Cartesian products
are immediate from finite semilinear descriptions, and coordinate projections are integral-affine images. The cited constructions are
effective, and there are only finitely many coset cells. Translations are integral affine on those cells. This proves \textup{(ii)}.

On fixed coset cells, multiplication has the form
\[
(r,u)(s,v)=(t,\,A_{r,s}u+B_{r,s}v+c_{r,s})
\]
for integral matrices and an integral translation vector determined by the finite quotient. Inversion and conjugation also follow in the same way. A fixed homomorphism has an integral linear part and a translation on each coset, after possibly refining cosets. There are only finitely many cosets, proving
\textup{(iii)}. Finally, non-emptiness is read from the finite semilinear description, proving \textup{(iv)}. For a systematic
treatment of effective rational constraints in virtually abelian groups, see \cite{CiobanuEvettsLevine}.
\end{proof}

We continue to state the local outputs as rational subsets and relations. In computations, Theorem~\ref{thm:rational-calculus} uses
coset-wise semilinear coordinates as the effective normal form for intersection and projection, and converts back to automata when a
rational-language output is required.

The preceding rational calculus turns the graph-of-groups conjugacy
criterion into the following abstract form of Pr\'eaux's reduction.

\begin{proposition}[Abstract Pr\'eaux reduction]\label{prop:abstract-preaux}
	Let \(\cG\) be a finite graph of groups whose edge groups are
	virtually \(\Z^2\), equipped with effective coordinates as in
	Definition~\ref{def:rational}, and whose edge monomorphisms are
	effectively given in those coordinates. Suppose:
	\begin{enumerate}[(i)]
		\item membership of every incident edge group in a vertex group is decidable;
		\item conjugacy and boundary parallelism are solvable in every vertex
		group, and the sets \(\mathcal C_E(u)\) and relations
		\(\mathcal P_{E,E'}(V)\) are effectively rational;
		\item every two-coset relation
		\(\mathcal C_{E,E'}(u,v)\) is effectively rational;
		\item there is a computable \(k\) such that any two conjugate non-trivial elliptic elements are joined in
		Theorem~\ref{thm:Collins}(ii) by a chain whose edge path has
		length at most \(k\).
	\end{enumerate}
	Then \(\pi_1(\cG)\) has solvable conjugacy problem.
\end{proposition}

\begin{proof}
	Conjugacy in the vertex groups gives their word problems. Together with \textup{(i)} and the effective edge maps, this gives the usual
	reduction and cyclic-reduction algorithms: after a positive	membership test, an edge-group preimage is found by enumeration.
	Use these algorithms to put the two input elements into cyclically reduced form.
	If their edge lengths differ, apply Theorem~\ref{thm:Collins}(i). If both have length zero, first deal
	with chains of length zero using the conjugacy algorithms in the vertex groups. The identity is conjugate only to itself. For
	non-trivial elements, \textup{(iv)} and finiteness of the quotient graph leave only finitely many edge-type sequences of length at most
	\(k\). Fix one such positive-length sequence
	\[
	(v_0,e_1,v_1,\ldots,e_p,v_p)
	\]
	and introduce variables \(c_i\in G_{e_i}\). Conjoin both endpoint
	conditions
	\[
	\iota^-_{e_1}(c_1)\sim_{G_{v_0}}u,
	\qquad
	\iota^+_{e_p}(c_p)\sim_{G_{v_p}}v,
	\]
	the intermediate passage conditions
	\[
	\iota^+_{e_i}(c_i)\sim_{G_{v_i}}
	\iota^-_{e_{i+1}}(c_{i+1})
	\qquad(1\leq i<p).
	\]
	The endpoint conditions are pullbacks of the boundary sets
	\[
	\mathcal C_{\iota^-_{e_1}(G_{e_1})}(u),
	\qquad
	\mathcal C_{\iota^+_{e_p}(G_{e_p})}(v),
	\]
	while the intervening conditions are the relations
	\(\mathcal P_{E,E'}(G_{v_i})\).
	Place the edge variables in the finite direct product
	\[
	P=\prod_{i=1}^pG_{e_i}
	\]
	and pull every condition back to \(P\) by the coordinate projections and fixed edge inclusions. The resulting subsets are rational by
	\textup{(ii)}, so Theorem~\ref{thm:rational-calculus} computes their intersection and decides whether it is empty. The conjugacies in
	the local relations quantify over all possible intermediate vertex-group labels. Note that there are usually infinitely many elements of \( \pi_1(\mathcal{G}) \) whose underlying edge path has length at most \( k\), but there are finitely many such edge paths; this calculation settles the question of conjugacy for each potential conjugator with that edge path pattern and so gives a terminating algorithm.

	If the common length is positive, there are only finitely many cyclic alignments of the underlying loops. For each alignment, put all edge
	variables into their finite direct product \(P\), pull back the rational two-coset relations in \textup{(iii)}, and intersect them
	with the rational edge-identification conditions, including the final condition closing the cycle. In the equations of
	Theorem~\ref{thm:Collins}(iii), the first boundary variable appears inverted; inversion is an integral-affine cell map by
	Theorem~\ref{thm:rational-calculus}\textup{(iii)}, so this pullback preserves effective rationality.
	Theorem~\ref{thm:rational-calculus} decides whether this intersection is empty, giving a finite terminating algorithm.
\end{proof}

\begin{remark}[Acylindricity and condition \textup{(iv)}]
\label{rem:acylindrical}
We call an action on a tree \(k\)-acylindrical if the pointwise stabiliser of every segment of length greater than \(k\) is finite.
Suppose that \(\pi_1(\cG)\) is torsion-free and its Bass--Serre action is \(k\)-acylindrical, for a known \(k\). Then condition
\textup{(iv)} holds with the same \(k\). Indeed, if \(1\neq u\in G_v\) and \(u^h\in G_w\), then \(u^h\) fixes the segment
joining \(v\cdot h\) to \(w\). Were this segment longer than \(k\), its pointwise stabiliser would be finite, and hence trivial by
torsion-freeness. Thus the segment has length at most \(k\) and supplies the chain in Theorem~\ref{thm:Collins}\textup{(ii)}.
\end{remark}

\section{The canonical elementary JSJ and its suspension}
\label{sec:GL}

\subsection{The Guirardel--Levitt tree}

Let \(G\) be a torsion-free one-ended hyperbolic group. We use the canonical elementary JSJ tree \(T_{\mathrm{can}}\), realised as the
tree of cylinders of the maximal-cyclic JSJ deformation space.

\begin{theorem}[Guirardel--Levitt, \cite{GuirardelLevittCylinders}, \cite{GuirardelLevittAutomorphisms}, \cite{GuirardelLevittJSJ}]\label{thm:GL}
The group \(G\) has a canonical \(\Out(G)\)-invariant cyclic JSJ tree \(T_{\mathrm{can}}\) with the following properties.
\begin{enumerate}[(i)]
\item The quotient \(T_{\mathrm{can}}/G\) is finite.
\item The tree is bipartite. Every edge joins a non-elementary vertex to an elementary vertex.
\item Every non-elementary vertex stabiliser is either rigid or QH.
\item Every elementary vertex stabiliser is maximal infinite cyclic.
Every edge stabiliser is maximal cyclic in its non-elementary
endpoint, and the stabilisers of two distinct edges incident at
a non-elementary vertex have trivial intersection.
\item At a QH vertex, the vertex group is the fundamental group of a compact hyperbolic surface with boundary, possibly
 non-orientable, and the incident edge groups are boundary subgroups.
\item Let \(R\) be a rigid vertex group, and let \(\mathcal P_R\) be the finite family of conjugacy classes in
 \(R\) of its incident edge groups. The subgroup of
 \(\Out(R;\mathcal P_R)\) induced by automorphisms of \(G\)
 preserving the orbit of \(R\) is finite. Here \(\Out(R;\mathcal P_R)\) denotes outer automorphisms preserving
 the finite family \(\mathcal P_R\) setwise; its permutation action records the finite permutations of incident edge orbits.
\end{enumerate}
\end{theorem}

\begin{proof}[References]
Trees of cylinders and their canonicity are developed in \cite{GuirardelLevittCylinders}. In the general relatively
hyperbolic setting, the canonical tree and its vertex types are
described in \cite[\S\S2.7 and 3.3]{GuirardelLevittAutomorphisms}.
In particular, the tree of cylinders is bipartite, with elementary
cylinder vertices and non-elementary rigid or QH vertices; see
\cite[pp.~614 and 618--619]{GuirardelLevittAutomorphisms}. Point \textup{(iv)} follows from
\cite[\S3.3 and \S4.1]{GuirardelLevittAutomorphisms}. There it is shown that incident edge stabilisers at a rigid or QH vertex are maximal elementary, and
distinct adjacent maximal elementary stabilisers have finite intersection, \cite[Corollary 3.2]{GuirardelLevittAutomorphisms}. In the present torsion-free hyperbolic setting these become maximal cyclic and trivial intersection (malnormal) statements. Finiteness of the induced outer action on a rigid vertex group follows
from \cite[Theorem~3.9 and \S4.1]{GuirardelLevittAutomorphisms}.
Since \(G\) is torsion-free, the finite fibre in the general definition of a QH vertex is trivial, and its orbifold group is actually surface group (in fact a free group). See also the
systematic account \cite{GuirardelLevittJSJ} and Bowditch's canonical splitting \cite{Bowditch}.
\end{proof}

If \(T_{\mathrm{can}}\) is trivial, then \(G\) is rigid or is a closed surface group. In the rigid case \(\Out(G)\) is finite, by
the same Bestvina--Paulin--Rips argument underlying Theorem~\ref{thm:GL}(vi). Thus, if \([\phi]\) has infinite order,
the trivial-tree case is precisely the surface case already discussed in the introduction.

\subsection{The induced mapping-torus action}

Canonicity supplies a \(\phi\)-equivariant tree automorphism
\(F:T_{\mathrm{can}}\to T_{\mathrm{can}}\), satisfying
\[
F(x\cdot g)=F(x)\cdot\phi(g).
\]
We extend the \(G\)-action to \(G_\phi\) by setting
\[
x\cdot t=F(x).
\]
After barycentric subdivision, if necessary, the action has no inversions. The tree \(T_{\mathrm{can}}\) is \(G\)-minimal, and a
minimal action of a finitely generated group on a simplicial tree is cocompact \cite[I.4]{Serre}. Hence \(T_{\mathrm{can}}/G\) is finite,
and \(F\) induces a \(\Z\)-action on this finite quotient.

Let \(\sigma\) be a vertex or edge of \(T_{\mathrm{can}}\), and let \(d_\sigma>0\) be the first-return time of its \(G\)-orbit under this
\(\Z\)-action. There is \(g_\sigma\in G\) such that
\[
\sigma\cdot t^{d_\sigma}=\sigma\cdot g_\sigma,
\]
so \(s_\sigma=t^{d_\sigma}g_\sigma^{-1}\) stabilises \(\sigma\). The restriction of the height map therefore gives a split exact sequence
\[
 1\longrightarrow G_\sigma\longrightarrow
 \Stab_{G_\phi}(\sigma)
 \xrightarrow{\ \chi\ } d_\sigma\Z\longrightarrow 1.
\]
Thus every vertex and edge stabiliser is the cyclic extension
\[
\Stab_{G_\phi}(\sigma) =G_\sigma\rtimes_{\theta_\sigma}\langle s_\sigma\rangle
\]
defined by its first-return automorphism
\(\theta_\sigma(h)=h^{s_\sigma}\). Changing \(s_\sigma\) changes
\(\theta_\sigma\) only by an inner automorphism. For a vertex \(x\), we continue to write
\[
V_x=\Stab_{G_\phi}(x).
\]

\begin{proposition}\label{prop:suspended-types}
The vertex and edge stabilisers of the \(G_\phi\)-action on \(T_{\mathrm{can}}\) have the following form.
\begin{enumerate}[(i)]
\item If \(G_x\cong\Z\) is elementary, then \(V_x\) is virtually
 \(\Z^2\).
\item If \(G_x\) is rigid, then \(Z(V_x)\cong\Z\) and
 \(V_x/Z(V_x)\) is hyperbolic.
\item If \(G_x\) is QH, then \(V_x\) is the fundamental group of a
 compact fibred three-manifold.
\item Every edge stabiliser is \( \Z \rtimes \Z\), hence either a torus or a Klein bottle. In particular, it is virtually \(\Z^2\).
\end{enumerate}
\end{proposition}

\begin{proof}
Apply the preceding first-return description. For an elementary vertex or an edge, \(G_\sigma\cong\Z\), and its first-return
automorphism is the identity or inversion. This gives (i) and (iv).

If \(G_x=R\) is rigid, Theorem~\ref{thm:GL}(vi) says that the first-return automorphism has finite order in \(\Out(R)\).
Proposition~\ref{prop:finite-outer} applied to \(R\) and this automorphism gives (ii).

If \(G_x=\pi_1(S)\) is QH, the first-return outer automorphism preserves the peripheral conjugacy classes, allowing permutation of
boundary components and inversion of their cyclic generators. The
type-preserving Dehn--Nielsen--Baer theorem for compact surfaces with boundary, including the non-orientable case, realises it by a
homeomorphism of \(S\); see \cite[\S3]{Fujiwara} and \cite[Theorem~1]{MaclachlanHarvey}. Its cyclic extension is the
fundamental group of the mapping torus of that homeomorphism, proving
(iii).
\end{proof}

\section{Refining suspended QH vertices}\label{sec:refinement}
\subsection{The equivariant refinement}

Let \(x\) be a QH vertex and write \(G_x=\pi_1(S)\), where \(S\) is a compact hyperbolic surface with boundary, not necessarily orientable.
The first-return automorphism determines a mapping class preserving the boundary components up to permutation. Before describing its
Nielsen--Thurston reduction, we illustrate the issue caused by a one-sided reduction curve.


\begin{example}[A one-sided reduction curve]
	\label{ex:mobius-extension}
	Let \(\Sigma\) be orientable and let \(f:\Sigma\to\Sigma\) be an
	orientation-preserving pseudo-Anosov homeomorphism fixing a boundary component pointwise. Attach a M\"obius band \(N\) along that component
	and extend \(f\) over \(N\) by the identity. The core \(c\) of \(N\) is then an invariant one-sided reduction curve.
	
	In the orientation double cover, the two lifts of \(\Sigma\) are joined by the annular lift of \(N\); the core of this annulus is an
	invariant essential two-sided curve. Thus the lifted monodromy is reducible, and the orientation double cover of the mapping torus has
	a non-trivial JSJ torus separating two pseudo-Anosov pieces. Cutting along \(\partial N\) isolates the M\"obius-band suspension as an
	elementary piece, with its boundary subgroup of index two.
\end{example}

\begin{lemma}[The surface reduction tree]\label{lem:surface-tree}
	Let \(x\) be a QH vertex of \(T_{\mathrm{can}}\), identify \(G_x=\pi_1(S)\), where \(S\) is a compact hyperbolic surface with boundary, possibly non-orientable, and let
	\[
	\theta_x(h)=h^{s_x}
	\]
	be the first-return automorphism defined above. Let \([f]\) be the mapping class corresponding to \([\theta_x]\) under the
	type-preserving Dehn--Nielsen--Baer theorem.
	
	There is a finite canonical reduction system \(\mathcal R_x=\mathcal R(f)\) and a representative \(f:S\to S\)
	such that \(f(\mathcal R_x)=\mathcal R_x\). Write \(\mathcal R_x^{(1)}\) and \(\mathcal R_x^{(2)}\) for its
	one-sided and two-sided subcollections. Choose disjoint M\"obius neighbourhoods \(N(c)\), equivariantly over
	\(\mathcal R_x^{(1)}\), and put
	\[
	\mathcal D_x=\mathcal R_x^{(2)}
	\cup\{\partial N(c):c\in\mathcal R_x^{(1)}\}.
	\]
	Thus \(\mathcal D_x\) is an \(f\)-invariant two-sided multicurve. Cutting along \(\mathcal D_x\) gives a finite graph of groups
	\(\Gamma_x\) with fundamental group \(\pi_1(S)\). A curve in \(\mathcal R_x^{(2)}\) gives the
	usual cyclic edge, while \(N(c)\) gives an elementary vertex
	\(\langle c\rangle\) and its boundary \(d=\partial N(c)\) gives the
	edge map
	\[
	d\longmapsto c^2. 
	\]
	Let \(T_x\) be the Bass--Serre tree of \(\Gamma_x\). The original mapping class induces an automorphism \(F_x:T_x\to T_x\) satisfying
	\[
	F_x(u\cdot g)=F_x(u)\cdot\theta_x(g)
	\qquad(g\in G_x,\ u\in T_x).
	\]
	The first-return mapping class on every non-elementary component orbit is periodic or pseudo-Anosov.
\end{lemma}

\begin{proof}
The non-orientable Nielsen--Thurston theorem follows from the orientable double cover; see
\cite[Theorem~25 and Propositions~27--29]{ParisNonorientable}, following Wu's canonical-reduction theorem \cite{Wu}. The canonical
system is preserved setwise by the unpowered mapping class. Homeomorphisms preserve sidedness, so
\(\mathcal R_x^{(1)}\) and \(\mathcal R_x^{(2)}\) are separately invariant and the neighbourhoods \(N(c)\) may be chosen equivariantly.

For each \(c\in\mathcal R_x^{(1)}\), cutting along \(d=\partial N(c)\) separates the M\"obius band \(N(c)\), with edge
map \(d\mapsto c^2\). The remaining non-elementary component first returns are periodic or pseudo-Anosov. A normal-form representative
permutes the components of this decomposition and therefore induces the asserted graph-of-groups and Bass--Serre-tree automorphisms.

The induced tree automorphism is equivariant with an automorphism representing \([\theta_x]\); composing it with the action of a suitable
element of \(G_x\) gives the stated \(\theta_x\)-equivariance. No purity assumption is required, since the representative may permute
components and curves. 

\end{proof}

\begin{remark}[Peripheral root extensions]
	\label{rem:root extensions}
		One could instead cut only along the two-sided curves in
	\(\mathcal R_x\); 
	\(\mathcal R_x^{(2)}\). Absorbing a M\"obius-band suspension \(A\) into an
	adjacent non-elementary piece \(V\) replaces \(V\) by
	\[
	V*_E A,
	\]
	where \(E\) is the boundary suspension and \([A:E]=2\). If \(V\) is
	pseudo-Anosov, then \(E\) is a full cusp subgroup, and
	\cite[Theorem~0.1\textup{(2)}]{DahmaniCombination} shows that
	\(V*_E A\) is hyperbolic relative to \(A\) and the remaining cusp groups. The analogous extension of a periodic piece is again
	finite-outer. Thus the later arguments could be reformulated using these finite peripheral root extensions. We instead cut along
	\(\partial N(c)\) in order to separate them into elementary vertices and retain the standard hyperbolic and Seifert-fibred pieces. This is largely to make 
	the references below apply more directly. It also means that our vertex groups are of familiar type, rather than root extensions of familiar type. 
\end{remark}

\begin{theorem}[Full-equivariant QH refinement]
\label{thm:pA-refinement}
The canonical tree admits a \(G_\phi\)-equivariant refinement
\[
T_{\mathrm{ref}}\longrightarrow T_{\mathrm{can}}
\]
obtained by blowing up every QH vertex along the cyclic splitting
dual to \(\mathcal D_x\), with the following properties:
\begin{enumerate}[(i)]
\item The \(G\)-action on \(T_{\mathrm{ref}}\) is minimal,
\(T_{\mathrm{ref}}/G\) is finite and, after barycentric
 subdivision, \(T_{\mathrm{ref}}\) is bipartite and the action
 has no inversions;
\item every edge stabiliser and every elementary vertex stabiliser is
 virtually \(\Z^2\);
\item every non-elementary vertex stabiliser is either a
pseudo-Anosov surface-by-\(\Z\) group or a finite-outer
hyperbolic-by-\(\Z\) group. At a pseudo-Anosov vertex, every
incident edge stabiliser is a full cusp subgroup of the vertex
stabiliser; 
\item no non-trivial element of \(G\) fixes two distinct edges
 incident at a non-elementary vertex.
\end{enumerate}
\end{theorem}

\begin{proof}
Fix a QH vertex \(x\) and use the \(G_x\)-tree \(T_x\) supplied by Lemma~\ref{lem:surface-tree}. The equivariance relation in that
lemma says exactly that the full vertex stabiliser
\[
  V_x=G_x\rtimes_{\theta_x}\langle s_x\rangle
\]
acts on \(T_x\), with \(g^{s_x}=\theta_x(g)\). Purity is not
needed: \(s_x\) simply permutes
component and curve vertices.

For every \(y=x\cdot m\) in the \(G_\phi\)-orbit of \(x\), transport
the tree \(T_x\), and hence its stabilisers, by right conjugation with
\(m\), so \(H\leq V_x\) is transported to \(H^m\). This is
independent of the choice of \(m\), since two choices differ by an
element of \(V_x\). Replacing each QH vertex by its transported tree
and attaching the old edge groups at their fixed vertices is the
standard tree-of-actions construction, and gives a \(G_\phi\)-tree.

The quotient \(T_{\mathrm{can}}/G\) is finite, and each quotient
\(T_x/G_x\) is the finite graph dual to \(\mathcal D_x\). Hence
\(T_{\mathrm{ref}}/G\) is finite.
The action is minimal: \(T_{\mathrm{can}}\) is \(G\)-minimal, and
each surface tree \(T_x\) is \(G_x\)-minimal, so the tree-of-actions
construction is \(G\)-minimal.

The stabiliser of a non-elementary component lift is the suspension of
its periodic or pseudo-Anosov first return. The surviving rigid
vertices are finite-outer by Theorem~\ref{thm:GL}\textup{(vi)}.

On every cyclic curve group first return acts by \(+1\) or \(-1\), so
its suspension is virtually \(\Z^2\). For
\(d=\partial N(c)\), the suspended inclusion into the elementary
M\"obius-band vertex retains the index two. At a pseudo-Anosov component, an incident edge stabiliser is the suspension
of the full stabiliser of the corresponding boundary-component orbit,
and hence is the full cusp subgroup of the mapping-torus group. Thus
the one-sided modification does not replace the cusp group at the pseudo-Anosov endpoint by a proper finite-index subgroup. Together
with Proposition~\ref{prop:suspended-types}\textup{(i)}, this proves
\textup{(ii)} and \textup{(iii)}. A barycentric subdivision removes
inversions and makes the refined tree bipartite.

Finally, at a surviving rigid vertex, part~\textup{(iv)} follows from Theorem~\ref{thm:GL}~\textup{(iv)}. At a surface-component vertex, incident
\(G\)-edge stabilisers are boundary subgroups of the component. These are maximal cyclic, and distinct boundary subgroups or distinct
conjugates of one have trivial intersection. This proves \textup{(iv)}.

Dahmani--Krishna perform essentially this refinement after passing to
a power which makes the mapping classes pure
\cite[\S2.4.2, especially pp.~407--408]{DahmaniKrishna}. The
argument above records why that convenience is unnecessary for the
full mapping-torus action.
\end{proof}

\section{Finite-outer folds and the block tree}
\label{sec:blocks}

The action of \(G_\phi\) on \(T_{\mathrm{ref}}\) need not be
acylindrical. In particular, this occurs if \( \phi \) has finite order as an outer automorphism, since then \( G_{\phi}\) has a centre which must act trivially on the tree. We are assuming that \( \phi\) has infinite order, but it can still happen that \( \phi \) has finite order on large `blocks' of \(T_{\mathrm{ref}}\). We already know that \( \phi\) has finite order on rigid suspensions and Seifert fibred suspensions; what can happen is that \( \phi \) may induce a finite order outer automorphism on the subgroup generated by two or more of these suspensions, and thus the corresponding central element will have a large fixed subtree. 

In fact, this is a local phenomenon: it occurs when two non-elementary pieces share a non-trivial central element along a common elementary piece. Our solution then is to fold the corresponding edges of \(T_{\mathrm{ref}}\) and iterate. When no more folds are possible, we show that the resulting tree is acylindrical. Since \(G_\phi\) is torsion-free,
Remark~\ref{rem:acylindrical} explains why this is the relevant
target, as it supplies condition~\textup{(iv)} of
Proposition~\ref{prop:abstract-preaux}.

\subsection{The block-reduction theorem}

The goal of the section is the following theorem.

\begin{theorem}[Block reduction]\label{thm:block-reduction}
There is a surjective simplicial \(G_\phi\)-map
\[
\beta:T_{\mathrm{ref}}\longrightarrow T_{\mathrm{blk}}
\]
with the following properties.
\begin{enumerate}[label=\textup{(\roman*)},ref=\textup{(\roman*)}]
\item The map \(\beta\) is a finite composite of full equivariant
 folds followed by the collapse onto the minimal subtree. Every edge
 of \(T_{\mathrm{ref}}\) maps either linearly onto an edge or to a
 vertex.
\item Both \(G\) and \(G_\phi\) act minimally and cocompactly on
 \(T_{\mathrm{blk}}\), and the \(G_\phi\)-action is
 four-acylindrical.
\item The tree \(T_{\mathrm{blk}}\) is bipartite, with elementary and
 non-elementary vertices. Every edge stabiliser has finite
 index in its adjacent elementary vertex stabiliser; in
 particular, both are virtually \(\Z^2\).
\item Every vertex is of exactly one of the following three types:
 \begin{enumerate}[label=\textup{(\alph*)}]
 \item an \emph{elementary vertex}, whose stabiliser is virtually
   \(\Z^2\);
  \item a \emph{pseudo-Anosov vertex}, whose stabiliser is
   \(\pi_1(S)\rtimes_\psi\Z\), where \(S\) is a compact
   hyperbolic surface and \(\psi\) is pseudo-Anosov;
  \item a \emph{finite-outer vertex}, whose stabiliser is
   \(H\rtimes_\theta\Z\), where \(H\) is non-elementary
   torsion-free hyperbolic and \([\theta]\) has finite order
   in \(\Out(H)\).
  \end{enumerate}
\item The star of every pseudo-Anosov vertex, including all its
  stabilisers, is unchanged from \(T_{\mathrm{ref}}\).
\item\label{item:block-local-G}
  If
  \(e_1,e_2\) are distinct edges incident at a non-elementary
  vertex, then
  \[
  G_{e_1}\cap G_{e_2}=1.
  \]
\end{enumerate}
\end{theorem}

The proof separates into three steps. We define intermediate trees and admissible folds, analyse one such fold, and then iterate until
none remains.
Four-acylindricity is a consequence of this terminality.

\subsection{Intermediate trees and admissible folds}

\begin{definition}[Intermediate tree]\label{def:intermediate-tree}
An \emph{intermediate tree} is a pair \((q,T)\), where
\[
     q:T_{\mathrm{ref}}\longrightarrow T
\]
is a surjective simplicial \(G_\phi\)-map which is a finite composite
of full equivariant folds and maps every edge of \(T_{\mathrm{ref}}\)
linearly onto an edge. We also require conclusions
\textup{(iii)--(vi)} of Theorem~\ref{thm:block-reduction} and the
cocompactness assertion in \textup{(ii)} to hold with \(T\) in place
of \(T_{\mathrm{blk}}\). Thus an intermediate tree has every property
required except, potentially, 
four-acylindricity and minimality. 
\end{definition}

For an intermediate tree \((q,T)\) and a vertex or edge \(\sigma\)
of \(T\), put
\[
 M_\sigma=\Stab_{G_\phi}(\sigma),\qquad
 G_\sigma=M_\sigma\cap G=\Stab_G(\sigma).
\]
Call a non-elementary vertex \(v\) \emph{finite-outer} if
\(G_v\) is hyperbolic and the first-return monodromy of
\(M_v\) on \(G_v\) has finite order in \(\Out(G_v)\); in other words \( v\) is called finite-outer when \( M_v\) is finite-outer, Definition~\ref{def: finite outer}. Put
\[
Z_v=Z(M_v).
\]
By Proposition~\ref{prop:finite-outer}, \(Z_v\) is infinite cyclic, the restriction of the height map \(\chi\) to \(Z_v\) is injective
and non-zero, and \(M_v/Z_v\) is hyperbolic. We use these standard consequences of finite-outerness without further comment.

The pair \((\mathrm{id},T_{\mathrm{ref}})\) is an intermediate tree.
Indeed, the identity map is the empty composite of full folds.
Theorem~\ref{thm:pA-refinement} gives the bipartition, the edge and vertex types, and its part~\textup{(iv)} gives
condition~\ref{item:block-local-G}. The edge and adjacent elementary vertex stabilisers are both virtually \(\Z^2\), so the edge inclusion
has finite-index image. The theorem also gives cocompactness of the \(G\)-action;
hence the \(G_\phi\)-action is cocompact as well. The assertions concerning pseudo-Anosov stars are tautological for the identity map.
Thus all the conditions of Definition~\ref{def:intermediate-tree} hold.

The following useful facts are automatic for every intermediate
tree.

\begin{lemma}[Suspensions and central incidence]
\label{lem:automatic-stabilisers}
Let \((q,T)\) be an intermediate tree.
\begin{enumerate}[label=\textup{(\roman*)}]
\item For every vertex or edge \(\sigma\) of \(T\), the stabiliser
  \(M_\sigma\) is a suspension of \(G_\sigma\).
\item If \(v\) is finite-outer and \(e\) is incident at \(v\), then
  \[
  Z_v\leq M_e,
  \]
  and the image of \(M_e\) in \(M_v/Z_v\) is virtually cyclic.
\end{enumerate}
\end{lemma}

\begin{proof}
Since \(G\lhd G_\phi\), the element \(t\) induces an action on the finite graph \(T/G\). Hence, for every vertex or edge \(\sigma\),
there are \(d>0\) and \(g\in G\) such that
\[
 \sigma\cdot t^d=\sigma\cdot g.
\]
Thus \(s=t^dg^{-1}\in M_\sigma\) has non-zero height. The height
map restricts to a split exact sequence
\[
 1\longrightarrow G_\sigma\longrightarrow M_\sigma
 \xrightarrow{\ \chi\ }\chi(M_\sigma)\longrightarrow1,
\]
where \(\chi(M_\sigma)\) is a non-zero subgroup of \(\Z\). Choosing an element of least positive height gives the required suspension.

Now let \(v\) be finite-outer, let \(e\) be incident at \(v\), and
let \(z\in Z_v\). Since \(z\) centralises \(G_v\),
\[
      G_{e\cdot z}=G_e^z=G_e.
\]
Moreover, \(G_e\neq1\): an edge of \(T_{\mathrm{ref}}\) mapping to \(e\) has a non-trivial cyclic \(G\)-stabiliser contained in
\(G_e\). If \(e\cdot z\neq e\), this equality therefore contradicts condition~\ref{item:block-local-G} for the intermediate
tree \(T\). Hence \(e\cdot z=e\), so \(Z_v\leq M_e\). Finally, \(M_e\) is virtually \(\Z^2\), and
quotienting it by the infinite cyclic subgroup \(Z_v\) gives a
virtually cyclic group.
\end{proof}

A \emph{full fold} at an elementary vertex \(a\) identifies two distinct incident edges pointwise, together with all their
\(G_\phi\)-translates. Equivalently, it is the elementary equivariant fold of \cite{BestvinaFeighn}. Its quotient is again a tree. 

\begin{definition}[Admissible fold]\label{def:admissible-fold}
Let \((q,T)\) be an intermediate tree, let \(a\) be an elementary vertex of \(T\), and let \(e_i=[a,v_i]\), \(i=1,2\), be distinct
incident edges whose opposite endpoints are finite-outer. Their full
fold is \emph{admissible} if
\[
Z_{v_1}\cap Z_{v_2}\neq1.
\]
\end{definition}

By Lemma~\ref{lem:automatic-stabilisers}\textup{(ii)}, both centres
lie in \(M_a\). 

We note that an admissible fold map is injective on elementary vertices and maps every edge linearly onto an edge.

The following is the induction step for the entire construction.

\begin{lemma}[One admissible fold]\label{lem:finite-outer-fold}
Let \((q,T)\) be an intermediate tree, let \(a\) be elementary, and
let \(e_i=[a,v_i]\), \(i=1,2\), be an admissible pair. Let
\[
 p:T\longrightarrow T'
\]
be their full equivariant fold. Then \((p\circ q,T')\) is again an intermediate tree.
\end{lemma}

\begin{proof}
Let \(e\) be the folded edge, let \(w\) be its non-elementary
endpoint, and put
\[
W=M_w,\qquad H=W\cap G.
\]
The standard stabiliser formulae for a full equivariant fold are as
follows. If \(e_1,e_2\) lie in distinct \(G_\phi\)-orbits, then
\[
 M_e=\langle M_{e_1},M_{e_2}\rangle,\qquad
 W=\langle M_{v_1},M_{v_2}\rangle.
\]
If \(e_2=e_1\cdot g\) for some \(g\in M_a\), then
\[
 M_e=\langle M_{e_1},g\rangle,\qquad
 W=\langle M_{v_1},g\rangle.
\]
Since \((q,T)\) is intermediate, \(M_{e_1}\) has finite index in
\(M_a\), while
\[
M_{e_1}\leq M_e\leq M_a.
\]
Thus \(M_e\) also has finite index in \(M_a\). All other edge and
elementary vertex stabilisers are unchanged.

Choose
\[
 1\neq z\in Z_{v_1}\cap Z_{v_2}.
\]
If the two edges lie in distinct orbits, the displayed formula for \(W\) immediately gives \(z\in Z(W)\). In the same-orbit case,
\(v_2=v_1\cdot g\) and \(Z_{v_2}=Z_{v_1}^{g}\). Both \(z\) and \(z^g\) belong to \(Z_{v_2}\), and
\(\chi(z^g)=\chi(z)\). Since \(\chi\) is injective on the infinite
cyclic group \(Z_{v_2}\), we have \(z^g=z\). The second displayed formula for \(W\) again gives
\[
 z\in Z(W),\qquad \chi(z)\neq0.
\]

The quotient tree \(T'\) is a cocompact \(G\)-tree. Its \(G\)-edge stabilisers are cyclic: this is unchanged away from the
folded orbit, while \(G_e=M_e\cap G\) contains \(G_{e_1}\) and is therefore infinite. It is virtually abelian and contained in the
torsion-free hyperbolic group \(G\), hence is infinite cyclic. Let \(T'_{\min}\) be the minimal \(G\)-subtree. The group \(H\)
contains \(G_{v_1}\), and hence is non-elementary. If \(w\notin T'_{\min}\), then \(H\) fixes both \(w\) and its projection
to \(T'_{\min}\), and therefore fixes an edge, contradicting the cyclicity of the edge stabilisers. Thus \(w\in T'_{\min}\).
Bowditch's vertex-group theorem \cite[Proposition~1.2]{Bowditch} now shows that \(H=\Stab_G(w)\) is
quasiconvex in \(G\), and hence hyperbolic.

Choose \(\tau\in W\) whose height generates \(\chi(W)=d\Z\), with
\(d>0\). Then \(W=H\rtimes_\theta\langle\tau\rangle\), where \(\theta\) is induced by \(\tau\). Since \(z\) has non-zero height,
we may write
\[
z=\tau^m b^{-1}
\]
for some \(m>0\) and \(b\in H\). Centrality of \(z\) says precisely that \(\theta^m=\ad_b\). Thus \([\theta]\) has finite order in
\(\Out(H)\), so \(w\) is a finite-outer vertex.

We next verify condition~\ref{item:block-local-G}. Away from the \(G_\phi\)-orbit of \(w\), the fold is injective on the
stars of non-elementary vertices, so condition~\ref{item:block-local-G} is inherited from \(T\). It
remains only to verify it at \(w\).

Write
\(p:T\to T'\) for the fold. If an edge \(e\) maps to \(e'\), then \(\Stab(e)\leq\Stab(e')\); both groups are virtually \(\Z^2\), so
this inclusion has finite index. Since the actions are cocompact, these indices are uniformly bounded over the finitely many edge
orbits. Hence, if \(h\) fixes \(e'\), some positive power \(h^N\), with \(N\) independent of the chosen preimage, fixes every edge in
\(p^{-1}(e')\).

Suppose now that \(1\neq h\in H\) fixes two distinct edges \(e'_1,e'_2\) at \(w\). Choose a common positive power \(h^N\)
which fixes every preimage of both edges. Let \(a_i\) be the elementary endpoint of \(e'_i\), and let \(\widetilde a_i\) be its
preimage in \(T\); a full fold is injective on elementary vertices, so \(\widetilde a_1\neq\widetilde a_2\). The element \(h^N\) fixes
the path between these vertices. That path contains an old non-elementary vertex at which \(h^N\) fixes two distinct incident
edges, contradicting condition~\ref{item:block-local-G} for \(T\).
Since \(G_\phi\) is torsion-free, \(h^N\neq1\), and the contradiction proves the claim.

All folds occur at elementary vertices and involve only finite-outer opposite endpoints. Consequently the elementary/non-elementary
bipartition, cocompactness, and the complete
stars of the pseudo-Anosov vertices are preserved. The new vertex is finite-outer and satisfies condition~\ref{item:block-local-G} by
the preceding argument. Every other vertex and edge retains the corresponding property from \(T\). The map
\(p\circ q\) is a surjective equivariant simplicial map, a finite composite of full folds, and maps every edge of
\(T_{\mathrm{ref}}\) linearly onto an edge. Thus
\((p\circ q,T')\) satisfies Definition~\ref{def:intermediate-tree}.
\end{proof}

Thus \(T_{\mathrm{ref}}\) has all the properties asserted in
Theorem~\ref{thm:block-reduction} except four-acylindricity. 
By Lemma~\ref{lem:finite-outer-fold} admissible folds may be iterated while preserving the defining properties of intermediate trees. Minimality and four-acylindricity will be dealt with after termination.

\begin{remark}\label{rem:twist-directions}
The admissibility criterion has a direct monodromy interpretation. For example, if
\[
 G=A*_C B,\qquad
 \phi|_A=\mathrm{id},\qquad
 \phi|_B=\ad_c\quad(c\in C),
\]
then the two local central directions in the suspended edge group \(\langle c,t\rangle\cong\Z^2\) are \(t\) and \(c^{-1}t\). Their
difference is the Dehn-twist parameter. The fold is admissible exactly when this twist disappears, so that the two locally inner
monodromies patch to an inner monodromy on the amalgam.
\end{remark}

\subsection{The terminal block tree}

%
%

\begin{proposition}[The terminal block tree]\label{prop:block-tree}
	Starting with \((\mathrm{id},T_{\mathrm{ref}})\), repeatedly perform an admissible full fold whenever one is available. This process terminates after finitely many steps. If
	\[
	\widehat\beta:T_{\mathrm{ref}}\longrightarrow T_{\mathrm{ter}}
	\]
	is the composite of a terminal sequence, then:
\begin{enumerate}[label=\textup{(\alph*)}]
	\item \((\widehat\beta,T_{\mathrm{ter}})\) is an intermediate tree;
	
	\item the minimal \(G\)-invariant subtree \(T_{\mathrm{blk}}\) of
	\(T_{\mathrm{ter}}\) is \(G_\phi\)-invariant, both \(G\) and
	\(G_\phi\) act minimally and cocompactly on it, and passing from
	\(T_{\mathrm{ter}}\) to \(T_{\mathrm{blk}}\) removes precisely the
	valence-one elementary vertices;
	
	\item if \(r:T_{\mathrm{ter}}\to T_{\mathrm{blk}}\) denotes this
	collapse, then
	\[
	\beta=r\circ\widehat\beta:T_{\mathrm{ref}}\longrightarrow
	T_{\mathrm{blk}}
	\]
	is a finite composite of full equivariant folds followed by a
	collapse, and every edge of \(T_{\mathrm{ref}}\) maps either linearly
	onto an edge or to a vertex;
	
	\item conclusions \textup{(iii)--(vi)} of
	Theorem~\ref{thm:block-reduction} hold for \(T_{\mathrm{blk}}\);
	
	\item \label{part:centres} if distinct finite-outer vertices \(v_1,v_2\) are adjacent to
	the same elementary vertex, then
	\[
	Z_{v_1}\cap Z_{v_2}=1.
	\]
\end{enumerate}
\end{proposition}
\begin{proof}
Every elementary vertex of an intermediate tree has finite valence:
there are finitely many incident edge orbits, and each edge stabiliser
has finite index in the elementary vertex stabiliser. There are also
only finitely many elementary vertex orbits. Thus
\[
 \kappa(T)=
 \sum_{[a]\in V_{\mathrm{elem}}(T)/G_\phi}
 \operatorname{val}_T(a)
\]
is finite. A full fold strictly decreases \(\kappa(T)\), so the
process terminates. The initial verification above supplies the
initial intermediate tree, and
Lemma~\ref{lem:finite-outer-fold} preserves this property at every
step. Hence \((\widehat\beta,T_{\mathrm{ter}})\) is intermediate, proving (a).

Let \(T_{\mathrm{blk}}\) be the minimal \(G\)-subtree of \(T_{\mathrm{ter}}\). Since \(G\) is normal in \(G_\phi\), uniqueness
of the minimal subtree makes \(T_{\mathrm{blk}}\) \(G_\phi\)-invariant. It is therefore minimal for both actions, and
cocompactness is inherited from \(T_{\mathrm{ter}}\).

Every non-elementary vertex \(v\) of \(T_{\mathrm{ter}}\) lies in
\(T_{\mathrm{blk}}\). Indeed, otherwise \(G_v\) would fix both \(v\)
and its projection to \(T_{\mathrm{blk}}\), and hence would fix an
edge, contrary to the cyclicity of the \(G\)-edge stabilisers. Thus only elementary vertices are removed on passing to the minimal
subtree. If an elementary vertex has at least two incident edges, their non-elementary endpoints lie in \(T_{\mathrm{blk}}\), so the
elementary vertex lies on the geodesic between them and also belongs to \(T_{\mathrm{blk}}\). Conversely, a minimal tree has no
valence-one vertices. Consequently the removed vertices are precisely the valence-one elementary vertices. For each such vertex \(a\), with
incident edge \(e\), we have \(M_a=M_e\), so this is a collapse of a redundant edge. This proves (b). 

No removed elementary vertex is adjacent to a pseudo-Anosov vertex. Indeed, the incident pseudo-Anosov edge is never folded. If its elementary endpoint has another incident edge in \(T_{\mathrm{ref}}\), then at least one such edge survives and remains distinct from the pseudo-Anosov edge. Otherwise that endpoint has valence one in \(T_{\mathrm{ref}}\), contrary to minimality.

Since \(T_{\mathrm{blk}}\) is a subtree of \(T_{\mathrm{ter}}\), no new vertices or edges are created, and the stabilisers of all retained vertices and edges are unchanged. Hence the bipartition, all vertex and edge
types, the complete stars of pseudo-Anosov vertices, and
condition~\ref{item:block-local-G} are inherited by
\(T_{\mathrm{blk}}\), proving (d). Let
\(r:T_{\mathrm{ter}}\to T_{\mathrm{blk}}\) be the equivariant
collapse and put \(\beta=r\circ\widehat\beta\). This proves \textup{(c)}.

Finally, if the displayed intersection were non-trivial, the two
incident edges already formed an admissible pair in
\(T_{\mathrm{ter}}\), contrary to terminality, proving (e).
\end{proof}

\begin{definition}
	\label{def:Zv generator}
	For each finite-outer vertex \(v\) of \(T_{\mathrm{blk}}\), let
	\(z_v\) be the generator of \(Z_v\) satisfying \(\chi(z_v)>0\).
	This choice is \(G_\phi\)-equivariant, since \(\chi\) is invariant
	under conjugation.
\end{definition}

\begin{remark}[Canonicity of the block tree]\label{rem:block-canonical}
	
Repeated applications of Lemma~\ref{lem:finite-outer-fold} serve to make the action `more' acylindrical, after which we pass to the minimal subtree to obtain the final block tree. This suggests that the terminal quotient is independent of the order of folds. However, we neither need nor claim that canonicity here.

We do prove that \( T_{\mathrm{blk}}\) is actually invariant under any automorphism of \( G_{\phi}\), Proposition~\ref{prop:invariant block tree}. 
\end{remark}

\subsection{Acylindricity}

We make explicit the peripheral fact used at pseudo-Anosov vertices.

\begin{lemma}[Pseudo-Anosov peripherals]\label{lem:pA-peripherals}
Let \(V\) be a pseudo-Anosov vertex stabiliser and let \(E\) stabilise
an incident edge. Put \(C=E\cap G\). Then
\[
        E=N_V(C).
\]
Moreover, stabilisers of distinct edges incident at \(V\) have
trivial intersection.
\end{lemma}

\begin{proof}

By Theorem~\ref{thm:pA-refinement}\textup{(iii)} and
Proposition~\ref{prop:block-tree}\textup{(d)}, the incident edge
groups are the full cusp groups of the corresponding finite-volume
hyperbolic mapping-torus group.

Such a group is strongly relatively hyperbolic relative to its cusp groups by
\cite[Theorem~5.1]{Farb}, and these peripherals are almost malnormal by \cite[Theorem~1.4]{Osin}; this is also the peripheral structure used
below for the local algorithms. The ambient group is torsion-free, so distinct peripheral conjugates have trivial intersection.

Certainly \(E\leq N_V(C)\). Conversely, if \(g\in N_V(C)\), then \(C\leq E\cap E^g\). Almost malnormality gives \(g\in E\), proving
the normaliser assertion. Distinct incident edges determine distinct peripheral conjugates: otherwise a non-trivial element of their edge
stabiliser would fix both edges, contrary to condition~\ref{item:block-local-G}. The final assertion now follows
from almost malnormality.
\end{proof}

\begin{lemma}[Four-acylindricity]\label{lem:four-acyl}
The pointwise stabiliser in \(G_\phi\) of a segment of
\(T_{\mathrm{blk}}\) containing more than four edges is trivial.
\end{lemma}

\begin{proof}
Suppose that \(1\neq g\in G_\phi\) fixes a segment containing five edges. Since the tree is bipartite, the segment contains two
non-elementary internal vertices \(V_1,V_2\), separated by an elementary vertex \(A\). If either \(V_i\) is pseudo-Anosov,
then \(g\) lies in the stabilisers of two distinct edges incident at \(V_i\), contrary to Lemma~\ref{lem:pA-peripherals}. Hence
\(V_i=M_{v_i}\), where both \(v_i\) are finite-outer vertices.

For each \(i\), the element \(g\) fixes the two distinct edges of the segment incident at \(v_i\). Condition~\ref{item:block-local-G}
therefore gives \(g\notin G_{v_i}\), and hence
\(\chi(g)\neq0\). Choose integers \(m_i>0\) and \(n_i\) such that
\[
m_i\chi(g)=n_i\chi(z_{v_i}).
\]
Then
\[
h_i=g^{m_i}z_{v_i}^{-n_i}
\]
has height zero. Both \(g\) and \(z_{v_i}\) fix the two incident edges. For the latter, this holds in \(T_{\mathrm{ter}}\) by
Lemma~\ref{lem:automatic-stabilisers}\textup{(ii)}, and is unchanged on passing to the minimal subtree. Thus \(h_i\in G_{v_i}\) fixes both edges, so
condition~\ref{item:block-local-G} gives \(h_i=1\). Consequently \(g^{m_i}\in Z_{v_i}\). Taking a common further power gives a
non-trivial element of \(Z_{v_1}\cap Z_{v_2}\), contrary to Proposition~\ref{prop:block-tree}~\ref{part:centres}.
\end{proof}

\begin{proof}[Proof of Theorem~\ref{thm:block-reduction}]
	
Proposition~\ref{prop:block-tree} gives conclusion \textup{(i)},
conclusions \textup{(iii)--(vi)}, and the minimality and cocompactness
assertions in \textup{(ii)}. Four-acylindricity is
Lemma~\ref{lem:four-acyl}.	
\end{proof}


\subsection{ \( T_{\mathrm{blk}}\) is \( \Aut(G_{\phi}) \)-invariant}

This subsection is not needed for the main argument, but is of independent interest. First, a preparatory lemma. We note that this is really \cite[Lemma~1.2]{BelegradekMostow}, but see also \cite[Theorem~2.2]{KoberdaHyperbolic}.

\begin{lemma}
	\label{lem:split hyperbolic abelian}
	Let \(P\) be the fundamental group of a finite-volume hyperbolic \(3\)-manifold. Then \(P\) admits no non-trivial splitting over a
	virtually abelian subgroup.
\end{lemma}

\begin{proof}
	We may first pass to the orientation cover \(P_0\) of \( P\). This is a normal subgroup of index at most 2 and therefore the induced splitting of \( P_0\) is non-trivial. If \(P_0 \) were to fix a vertex in the tree, that vertex would be unique as \(P_0\) is not virtually abelian and hence must also be fixed by \(P\).  Hence we have a non-trivial splitting of \(P_0\) over a virtually abelian group. (Actually, the splitting will be over abelian groups, having converted Klein bottles to tori). 
	
%
	
	This contradicts \cite[Lemma~1.2]{BelegradekMostow}, which proves
	that the fundamental group of a complete finite-volume pinched
	negatively curved manifold of dimension greater than two admits no
	non-trivial splitting over a virtually nilpotent subgroup.
\end{proof}

\begin{proposition}[\(\Aut(G_{\phi}) \)-invariance of \( T_{\mathrm{blk}}\)]
	\label{prop:invariant block tree}
	
	Under the natural identification \( G_{\phi} \cong \Inn(G_{\phi})\),  the action of \( G_{\phi}\) on \( T_{\mathrm{blk}}\) extends to a simplicial action of \( \Aut(G_{\phi})\) on \( T_{\mathrm{blk}}\). 
\end{proposition}

%

\begin{proof}
	Throughout this proof we use the fact that a finitely generated
	group acting on a tree either admits a non-empty fixed subtree or
	a unique minimal invariant subtree, on which it acts cocompactly.
	Since we are using right actions, we will let automorphisms act on the right for consistency.

	Let \(\alpha\in\Aut(G_{\phi})\). We will show that any
	non-elementary vertex group of \(T_{\mathrm{blk}}\) is sent to
	another, preserving type. 
	
	\medspace
	
	First consider a pseudo-Anosov vertex group \(P\), and put
	\(Q=P\alpha\). Every edge stabiliser for the induced action of \(Q\)
	on \(T_{\mathrm{blk}}\) is virtually abelian. Hence, if \(Q\) were
	not elliptic, its minimal invariant subtree would give a non-trivial
	splitting of \(Q\) over a virtually abelian subgroup, contrary to
	Lemma~\ref{lem:split hyperbolic abelian}. Therefore \(Q\) fixes a point of
	\(T_{\mathrm{blk}}\). Since \(Q\) is non-elementary, this fixed point
	is unique.

	Next suppose that \(V\) is a finite-outer vertex group of
	\(T_{\mathrm{blk}}\) and put \(U=V\alpha\). Again let \(U\) act
	on \(T_{\mathrm{blk}}\). Recall that the centre of \(U\) is
	infinite cyclic, generated by \(z\). If \(z\) acts hyperbolically
	on \(T_{\mathrm{blk}}\), then its axis \(L\) is preserved by \(U\)
	by normality. The kernel of the map
	\[
	U\longrightarrow\operatorname{Isom}(L)
	\]
	would be infinite, as \(U\) is not virtually cyclic,
	contradicting acylindricity. Therefore \(z\) acts elliptically and
	its fixed subtree has finite diameter by acylindricity, is
	preserved by \(U\) by normality, and therefore \(U\) has a bounded
	orbit. This implies that \(U\) is elliptic, fixing a point in
	\(T_{\mathrm{blk}}\). As before, \(U\) must fix a unique point
	since it is non-elementary.
	
	This shows that every non-elementary vertex group is sent by
	\(\alpha\) to a unique non-elementary vertex group. Applying the
	same argument to \(\alpha^{-1}\) shows that \(\alpha\) permutes
	the non-elementary vertex groups.
	
	Next we deal with elementary vertices. We first observe that a
	virtually \(\Z^2\) subgroup can fix at most one elementary vertex.
	Indeed, if such a group fixed two distinct elementary vertices,
	then a non-trivial element of its intersection with \(G\) would
	fix two distinct edges incident at a non-elementary vertex,
	contrary to Theorem~\ref{thm:block-reduction} \textup{(vi)}.
	
	Now let \(w\) be an elementary vertex. By minimality of
	\(T_{\mathrm{blk}}\), it has distinct non-elementary neighbours
	\(u,v\). Since the stabilisers of the two incident edges have
	finite index in \(M_w\), the group
	\[
	K=M_u\cap M_v
	\]
	has finite index in \(M_w\). By the preceding arguments, there
	are non-elementary vertices \(u',v'\) such that
	\[
	M_u\alpha=M_{u'}
	\qquad\text{and}\qquad
	M_v\alpha=M_{v'}.
	\]
	Thus
	\[
	K'=K\alpha=M_{u'}\cap M_{v'}
	\]
	is virtually \(\Z^2\) and has finite index in \(M_w\alpha\). If
	the distance between \(u'\) and \(v'\) were greater than \(2\),
	then \(K'\) would fix more than one elementary vertex,
	contradicting the argument above. Hence the distance between
	\(u'\) and \(v'\) is exactly \(2\), and there is a unique
	elementary vertex \(w'\) between them.
	
	Let \(N\) be the core of \(K'\) in \(M_w\alpha\). This is a
	finite-index normal subgroup of \(M_w\alpha\) which is virtually
	\(\Z^2\) and fixes \(w'\). In particular, this is the only
	elementary vertex it fixes. Since \(N\) is normal in
	\(M_w\alpha\), the subtree \(\operatorname{Fix}(N)\) is
	\(M_w\alpha\)-invariant. Moreover, \(w'\) is the unique
	elementary vertex in \(\operatorname{Fix}(N)\). Hence
	\(M_w\alpha\) fixes \(w'\), so
	\[
	M_w\alpha\leq M_{w'}.
	\]
	Since \(M_{w'}\alpha^{-1}\) is contained in a unique elementary
	vertex stabiliser and also contains \(M_w\), we deduce that
	\[
	M_w\alpha=M_{w'}.
	\]
	
	Therefore \(\alpha\) permutes the vertex stabilisers of
	\(T_{\mathrm{blk}}\). It remains to recover adjacency. If \(v\)
	is non-elementary and \(w\) is elementary, then we claim that
	\[
	d(v,w)=1
	\quad\Longleftrightarrow\quad
	M_v\cap M_w\text{ is virtually }\Z^2.
	\]
	One direction follows because the edge stabiliser has finite
	index in the elementary stabiliser \(M_w\). Conversely, if \(v\)
	and \(w\) are not adjacent, then the path between them contains
	two distinct edges incident at an internal non-elementary vertex.
	Since \(M_v\cap M_w\) being virtually \(\Z^2\) implies that
	\(G_v\cap G_w\) is infinite, this would contradict
	Theorem~\ref{thm:block-reduction} \textup{(vi)}. This proves the
	claim.
	
	For each vertex \(v\), let \(v f_\alpha\) be the unique vertex
	such that
	\[
	M_v\alpha=M_{v f_\alpha}.
	\]
	The discussion above shows that \(f_\alpha\) is a bijection on
	the vertices which preserves adjacency, and hence is an
	automorphism of \(T_{\mathrm{blk}}\).
	
	For \(\alpha,\beta\in\Aut(G_{\phi})\), we have
	
	\[
		M_{v f_\alpha f_\beta} =(M_{v f_\alpha})\beta =(M_v\alpha)\beta =M_v(\alpha\beta) =M_{v f_{\alpha\beta}}.
	\]
	
	By uniqueness, \(f_\alpha f_\beta=f_{\alpha\beta}\), while
	\(f_{\mathrm{id}}=\mathrm{id}\). Thus
	\(\alpha\mapsto f_\alpha\) defines a right action of
	\(\Aut(G_{\phi})\) on \(T_{\mathrm{blk}}\).
	
	Finally, for \(\gamma\in G_{\phi}\),
	\[
	M_{v f_{\operatorname{ad}(\gamma)}}
	=
	M_v\operatorname{ad}(\gamma)
	=
	M_v^\gamma
	=
	M_{v\cdot\gamma}.
	\]
	Hence, again by uniqueness,
	\[
	v f_{\operatorname{ad}(\gamma)}=v\cdot\gamma.
	\]
	Thus the action of \(\Inn(G_{\phi})\) agrees with the original
	action of \(G_{\phi}\).
\end{proof}

\begin{remark}
	By the usual arguments from \cite{CullerMorgan}, Proposition~\ref{prop:invariant block tree} is really equivalent to saying that \( \Aut(G_{\phi})\) preserves the translation length function of \( T_{\mathrm{blk}}\). 
\end{remark}

\section{Verification for the acylindrical splitting}\label{sec:local-verification}

The goal of this section is to verify the hypotheses of
Proposition~\ref{prop:abstract-preaux} for the graph of groups
\[
\cG_{\mathrm{blk}}=T_{\mathrm{blk}}/G_\phi.
\]

Theorem~\ref{thm:block-reduction}~\textup{(ii)} and
Remark~\ref{rem:acylindrical} already verify hypothesis
\textup{(iv)}, with \(k=4\). So we are left with hypotheses \textup{(i)--(iii)}. We proceed one vertex type at a time.


\subsection{Elementary vertices}

\begin{proposition}[Local problems at an elementary vertex]
\label{prop:elementary-package}
Let \(V\) be an elementary vertex group of \(\cG_{\mathrm{blk}}\), and let \(E,E'\leq V\) be incident edge-group images. Word and conjugacy
in \(V\), and membership in \(E\), are decidable. Given \(u,v\in V\), the sets
\(\mathcal C_E(u)\), \(\mathcal P_{E,E'}(V)\), and
\(\mathcal C_{E,E'}(u,v)\) are effective. The three displayed sets
are effectively rational.
\end{proposition}

\begin{proof}
All these sets become finite systems of affine equations, congruences, and finite-quotient conditions after splitting into the
coset cells of Theorem~\ref{thm:rational-calculus}. Their solution sets are effective semilinear sets, and part~\textup{(i)} of that
theorem converts them to the required rational descriptions.
\end{proof}

\subsection{Pseudo-Anosov vertices}

We next treat the hyperbolic three-manifold pieces. The argument is
essentially due to Pr\'eaux; we use relative hyperbolicity to obtain
the rational descriptions required by
Definition~\ref{def: local problems}.

\begin{proposition}[Local problems for hyperbolic three-manifolds]
\label{prop:effective-preaux}
Let \(V=\pi_1(S)\rtimes_\psi\Z\) be a pseudo-Anosov vertex group of \(\cG_{\mathrm{blk}}\), and let \(E,E'\leq V\) be incident edge
groups, with the fixed data of Remark~\ref{rem:fixed-input}. Word and
conjugacy in \(V\), and membership in every incident edge group, are decidable. Given words
\(u,v\in V\), one can construct automata, equivalently coset-wise
semilinear descriptions, for every set \(\mathcal C_E(u)\), every relation \(\mathcal P_{E,E'}(V)\), and every two-coset relation
\(\mathcal C_{E,E'}(u,v)\). These conclusions also hold when \(S\) or its mapping torus is non-orientable.
\end{proposition}

\begin{proof}
The analysis below is essentially due to Pr\'eaux. In the orientable case, Proposition~4.1 and Lemmas~4.2--4.3 of
\cite{PreauxOrientable} give the required finite and affine families, while Theorem~6.3 and Proposition~6.2 compute an initial
solution whenever the set is non-empty. The non-orientable word and conjugacy algorithms are supplied by \cite[\S\S2--4]{PreauxNonorientable}.
For completeness, we now give a relative-hyperbolic argument which also proves the stronger rational-output
statement in the non-orientable case.

The incident edge groups are the full cusp subgroups described in
Lemma~\ref{lem:pA-peripherals}. By \cite[Theorem~5.1]{Farb}, \(V\) is (strongly) hyperbolic relative to its
finite family of full cusp groups. Each cusp group is relatively
quasiconvex directly from \cite[Definition~1.8]{Osin}: any two of
its elements are joined by a single peripheral edge in the relative
Cayley graph. The peripheral family is almost malnormal by
\cite[Theorem~1.4]{Osin}. Since \(V\) is torsion-free,
distinct maximal peripheral subgroups have trivial intersection and
each maximal peripheral subgroup is self-normalising. Almost malnormality of
the peripheral structure shows that \(E\) has peripherally finite
index in the sense of \cite[\S7.3]{KharlampovichMiasnikovWeil}. These peripheral/cusp subgroups are virtually
\(\Z^2\), and hence geodesically biautomatic. 
Each incident edge group \(E\) belongs
to the peripheral structure of \(V\).

Theorem~7.5 and Corollary~7.9 of \cite{KharlampovichMiasnikovWeil}, followed by the algorithms of their
Section~6, therefore give effective membership, coset and double-coset languages, and intersections for these subgroups. The
peripheral-conjugacy test used below is supplied by \cite[Corollary~5.14]{Bumagin}. When it is positive, an initial pair
\(e_0\in E\), \(g_0\in V\) with \(u=e_0^{g_0}\) is found by enumeration using the word problem.

We now verify that the local problem sets of Definition~\ref{def: local problems} are rational. 

\medspace

If \(u \in V\) is not conjugate into \(E\), then \(\mathcal C_E(u)\) is empty. Otherwise choose \(e_0\in E\) and \(g_0\in V\) with \(u=e_0^{g_0}\). If
\(e^g=u\) for \(e\in E\), then the maximal peripherals \(E^g\) and \(E^{g_0}\) have non-trivial intersection and hence are equal.
Self-normalisation gives \(g=hg_0\) for some \(h\in E\), and therefore
\[
 \mathcal C_E(u)=\{e\in E\mid e\sim_Ee_0\}.
\]
For \(u=1\) the set is \(\{1\}\). Internal conjugacy in the virtually abelian group \(E\) is effectively rational by
Theorem~\ref{thm:rational-calculus}.

The same argument treats the two-boundary relation. The fixed cusp data determine whether \(E^{g_0}=E'\) for some \(g_0\), and provide
such a \(g_0\) when it exists. If it does not exist, then \(\mathcal P_{E,E'}(V)=\{(1,1)\}\). If it does, then
\[
 \mathcal P_{E,E'}(V)
 =\{(e,e')\in E\times E'\mid
     e'\sim_{E'}e^{g_0}\},
\]
is an affine image of the internal conjugacy relation in \(E'\), and
hence effectively rational.

Finally, Proposition~6.2 of \cite{KharlampovichMiasnikovWeil} decides whether \(u\in EvE'\).
If not, the two-coset relation is empty. If
\(u=e_0ve'_0\), put
\[
       I=E\cap(E')^{v^{-1}},
\]
which is effectively computed by Corollary~7.9 of that paper.

We claim that
\[
 \mathcal C_{E,E'}(u,v)
 =\left\{\left(e_0c,(c^{-1})^ve'_0\right)\ \middle|\ c\in I\right\}.
\]

Let \((e,e')\) be any solution and write \(e=e_0c\), where \(c\in E\). Comparing \(eve'=u=e_0ve'_0\) and cancelling \(e_0\)
gives
\[
cve'=ve'_0,
\qquad\text{hence}\qquad
c^ve'=e'_0.
\]
Thus
\[
e'=(c^{-1})^ve'_0.
\]
Since \(e',e'_0\in E'\), we have \(c^v\in E'\), so
\(c\in E\cap(E')^{v^{-1}}=I\). Conversely, every \(c\in I\)
gives a solution, since
\[
(e_0c)v\bigl((c^{-1})^ve'_0\bigr)=e_0ve'_0=u.
\]
Almost malnormality says that \(I\) is either trivial or a full maximal peripheral subgroup. In either case the displayed affine
image is effectively rational by Theorem~\ref{thm:rational-calculus}.

\end{proof}

\subsection{Finite-outer vertices}

It remains to treat the finite-outer vertices. For each such vertex
\(x\), put
\[
V=M_x,\qquad C=Z(V)=\langle z_x\rangle,\qquad
q:V\longrightarrow Q=V/C,
\]
where \(z_x\) is the positive generator chosen in Definition~\ref{def:Zv generator}. The group \(Q\) is hyperbolic by
Proposition~\ref{prop:finite-outer}.

The local arguments first solve the relevant problem in \(Q\) and then use the height homomorphism to eliminate the central error.

\subsubsection*{\normalfont\bfseries Fixed local data}

In accordance with Remark~\ref{rem:fixed-input}, the fixed-pair data
include a word for \(z_x\), a presentation and chosen automatic
structure for \(Q\), the required hyperbolicity constants, and the
images under \(q\) of the chosen generators of \(V\).
Lemma~\ref{lem:automatic-stabilisers}\textup{(ii)} gives \(C\leq E\)
for every incident edge group \(E\). Since \(E\) is virtually
\(\Z^2\), the quotient
\[
q(E)=E/C
\]
is virtually cyclic. For every such \(E\), we fix the edge monomorphism \(E\to V\), the coordinate form of \(q|_E\), a normal
cyclic subgroup of finite index in \(q(E)\), a finite transversal, and lifts in \(E\) of the transversal and cyclic generator. Hence
all changes between edge, quotient, height, and central coordinates are integral-affine on finitely many coordinate cosets.

Although \(G\) and \(G_\phi\) are torsion-free, the quotient \(Q\) need not be. We shall use the following elementary calculation to
recover coordinates for subgroups of the virtually cyclic groups \(q(E)\).

\begin{lemma}[Effective subgroups of virtually cyclic groups]
	\label{lem:VC-coordinates}
	Let \(U\) be a fixed virtually cyclic group with the coordinate data described above. From words generating a subgroup \(J\leq U\), one
	can compute \(J\) as an effectively rational subset of \(U\), and hence obtain its fixed coset-wise coordinates.
\end{lemma}

\begin{proof}
	First rewrite the given generators as words in the fixed generators of \(U\). Since they are known to lie in \(U\), such expressions can
	be found by enumeration using the word problem. An automaton with loops labelled by these words and their inverses accepts precisely
	\(J\). Thus \(J\) is effectively rational, and Theorem~\ref{thm:rational-calculus}\textup{(i)} computes its
	coset-wise coordinate description.
\end{proof}

Thus, when a quasiconvex-subgroup algorithm below returns generators for an intersection, Lemma~\ref{lem:VC-coordinates} computes its
coordinate inclusions into the ambient virtually cyclic groups.

\subsubsection*{\normalfont\bfseries Effective virtually cyclic calculus}

The next lemma follows the same general strategy as Proposition~\ref{prop:effective-preaux}, with virtually cyclic
subgroups of a hyperbolic group replacing cusp subgroups of a relatively hyperbolic group. Since torsion and finite intersections
introduce somewhat different cases, we give the details.

\begin{lemma}[Effective virtually cyclic calculus]
	\label{lem:VC-calculus}
	Let \(Q\) be a hyperbolic group and let \(A,B\leq Q\) be virtually
	cyclic quasiconvex subgroups with the coordinate data described above.
	For \(x,y\in Q\), the three local solution sets
	\[
	\mathcal C_A(x)\subseteq A,\qquad
	\mathcal P_{A,B}(Q)\subseteq A\times B,\qquad
	\mathcal C_{A,B}(x,y)\subseteq A\times B
	\]
	are effectively rational in the fixed coordinates.
\end{lemma}

\begin{proof}
	We first record the algorithms being used. Propositions~6.1, 6.2, and 6.4 of \cite{KharlampovichMiasnikovWeil} give membership,
	containment, coset, double-coset, and intersection algorithms for quasiconvex subgroups of a hyperbolic group. Propositions~6.7--6.9
	and Corollary~6.10 give the corresponding bounded-representative algorithms, using bounded packing. These procedures terminate here
	because \(A\) and \(B\) are known to be quasiconvex. The word and conjugacy problems and a shortlex automatic structure for \(Q\) are
	effective; see \cite[Chapter~3]{Epstein}. We use these algorithms below to obtain the stronger, parametric descriptions required here.
	
	We shall also use effective internal conjugacy in \(A\) and \(B\). Let \(U\) denote either of these groups. If \(U\) is finite, its
	conjugacy relation can be computed by enumeration. Otherwise write
	\[
	U=\bigcup_{r\in R}r\langle a_U\rangle
	\]
	using the fixed finite transversal \(R\) and normal cyclic subgroup
	\(\langle a_U\rangle\). For fixed \(r,s,t\in R\), substitute
	\[
	u=ra_U^m,\qquad u'=sa_U^n,\qquad h=ta_U^k
	\]
	into the equation \(u^h=u'\). The finite multiplication table and the action of \(R\) on \(\langle a_U\rangle\) reduce this equation to
	integral-affine equations and congruence conditions in \(m,n,k\). Taking the finite union over \(r,s,t\), and then projecting away the
	coordinate \(k\), gives an effective rational description of
	\[
	\Conj_U=\{(u,u')\in U^2\mid u\sim_Uu'\}.
	\]
	This calculation includes the finite kernel and the elements mapping to reflections in the dihedral case; no infinite family is enumerated
	element by element.
	
	We first compute \(\mathcal P_{A,B}(Q)\). By the standard structure of virtually cyclic groups
	\cite[III.\(\Gamma\).3.10]{BridsonHaefliger}, the torsion elements of \(A\) and \(B\) lie in finitely many internal conjugacy classes.
	The preceding coordinate calculation computes representatives and effective rational descriptions of these classes. For each pair of
	representatives, use the conjugacy algorithm in \(Q\) to decide whether they are conjugate in \(Q\). The union of the corresponding
	products of internal conjugacy classes is exactly the torsion part of \(\mathcal P_{A,B}(Q)\), and is effectively rational.
	
	It remains to consider pairs of infinite-order elements. Propositions~6.7--6.9 and Corollary~6.10 of
	\cite{KharlampovichMiasnikovWeil} give a computable finite list \(g_1,\ldots,g_s\in Q\) representing every double coset
	\[
	AgB\qquad\text{for which}\qquad B\cap A^g
	\text{ is infinite}.
	\]
	Concretely, the cited bound reduces the search to a finite ball, and Proposition~6.4 computes the relevant intersections.

	For each remaining representative put
	\[
	I_j=A\cap B^{g_j^{-1}}.
	\]
	The subgroup \(I_j\) is infinite and virtually cyclic, and
	\[
	I_j^{g_j}=A^{g_j}\cap B.
	\]
	The intersection algorithm computes generators for \(I_j\), and hence for \(I_j^{g_j}\). Lemma~\ref{lem:VC-coordinates} computes their
	coordinates in \(A\) and \(B\), respectively, and shows that the graph
	\[
	\{(c,c^{g_j})\mid c\in I_j\}
	\subseteq A\times B
	\]
	is effectively rational. Define
	\[
	\begin{split}
		R_j=\{(a,b)\in A\times B\mid {}
		\exists \text{\(c\in I_j\) such that} \ 
		a\sim_Ac\ \text{and}\ b\sim_Bc^{g_j}\}.
	\end{split}          \tag{\(*\)}
	\]

Therefore \( R_j\) is the projection of the set, 
\[
\{ (a,c,c^{g^j},b) \ : \   a\sim_Ac\ \text{and}\ b\sim_Bc^{g_j}\ \},
\] 
to 	\(A\times B\). Since conjugacy (in \( A\) and \( B\)), conjugation by \( g_j\) intersection and projection are effectively rational, we conclude that \(R_j\) is effectively rational, by Theorem~\ref{thm:rational-calculus}.

	We verify that these are all the infinite-order pairs in \(\mathcal P_{A,B}(Q)\). Suppose that \(b=a^g\), where
	\(a\in A\) and \(b\in B\) have infinite order. Then
	\[
	\langle b\rangle\leq B\cap A^g,
	\]
	so \(g=\alpha g_j\beta\) for some \(\alpha\in A\), \(\beta\in B\), and some \(j\). Putting
	\(c=a^\alpha\), we have
	\[
	c\in I_j,\qquad
	b=(c^{g_j})^\beta.
	\]
	Hence \((a,b)\in R_j\). Conversely, the two internal conjugacies in \((*)\), together with conjugation by \(g_j\), show that every pair
	in \(R_j\) is conjugate in \(Q\). The torsion relation together with the finitely many relations \(R_j\) is therefore precisely \(\mathcal P_{A,B}(Q)\).

	To compute \(\mathcal C_A(x)\), apply the preceding construction with \(B=\langle x\rangle\). If \(x\) has finite order, its cyclic group
	and coordinate data are computed by enumeration, using the computable torsion bound for \(Q\). If \(x\) has infinite order, use \(x\) as
	the cyclic generator. In either case, \(\langle x\rangle\) is quasiconvex and the required quasiconvexity data are effectively
	computable from \(x\) using the fixed hyperbolic structure on \(Q\). The preceding construction therefore applies to this variable
	subgroup. Intersect \(\mathcal P_{A,\langle x\rangle}(Q)\) with \(A\times{x}\) and project to \(A\). The result is exactly \(\mathcal C_A(x)\).

	Finally, consider \(\mathcal C_{A,B}(x,y)\).
	Proposition~6.2 of \cite{KharlampovichMiasnikovWeil} decides whether \(x\in AyB\). If not, the relation is empty. Otherwise find
	\(a_0\in A\) and \(b_0\in B\) such that
	\[
	x=a_0yb_0,
	\]
	and put
	\[
	I=A\cap B^{y^{-1}}.
	\]
	Proposition~6.4 computes \(I\). A direct calculation gives
	\[
	\mathcal C_{A,B}(x,y)
	=\left\{\left(a_0c,(c^{-1})^yb_0\right)
	\ \middle|\ c\in I\right\}.
	\]
	Indeed, if \(ayb=a_0yb_0\), then
	\[
	c=a_0^{-1}a
	=(b_0b^{-1})^{y^{-1}}\in A\cap B^{y^{-1}},
	\]
	and the displayed formula follows; the converse is immediate.
	
	Lemma~\ref{lem:VC-coordinates} gives an effective coordinate description of \(I\). On each of its finitely many coordinate
	cosets, the map
	\[
	c\longmapsto
	\left(a_0c,(c^{-1})^yb_0\right)
	\]
	is integral affine. Its image is therefore effectively rational by Theorem~\ref{thm:rational-calculus}. This proves the result.
\end{proof}

\subsubsection*{\normalfont\bfseries Membership, conjugacy, and boundary relations}

\begin{lemma}\label{lem:rigid-membership}
Let \(x\) be a finite-outer vertex of \(T_{\mathrm{blk}}\) and put
\[
 V=M_x,\qquad z=z_x,\qquad C=\langle z\rangle,\qquad
 q:V\longrightarrow Q=V/C,
\]
and let \(E\leq V\) be an incident edge group. Membership in \(E\) is decidable.
\end{lemma}

\begin{proof}
The subgroup \(q(E)\) is virtually cyclic, hence quasiconvex in the hyperbolic group \(Q\), and membership in it is decidable by
\cite[Proposition~6.1]{KharlampovichMiasnikovWeil}. Given \(v\in V\), decide whether \(q(v)\in q(E)\). If not, then
\(v\notin E\). If it does, choose \(e\in E\) with \(q(e)=q(v)\). Then \(ve^{-1}\in C\leq E\), so \(v\in E\).
\end{proof}

\begin{lemma}\label{lem:rigid-conjugacy}
Let \(x\) be a finite-outer vertex of \(T_{\mathrm{blk}}\) and put
\[
 V=M_x,\qquad z=z_x,\qquad C=\langle z\rangle,\qquad Q=V/C.
\]
Then the conjugacy problem in \(V\) is solvable.
\end{lemma}

\begin{proof}
Apply Lemma~\ref{lem:height} to the chosen central subgroup \(C=\langle z\rangle\) and the hyperbolic quotient \(Q=V/C\).
\end{proof}

\begin{proposition}\label{prop:rigid-boundary}
Let \(x\) be a finite-outer vertex of \(T_{\mathrm{blk}}\) and put
\[
 V=M_x,\qquad z=z_x,\qquad C=\langle z\rangle,\qquad
 q:V\longrightarrow Q=V/C,
\]
and let \(E\leq V\) be an incident edge group. For every \(u\in V\), the boundary-parallelism problem is solvable and
\(\mathcal C_E(u)\) is an effectively rational subset of \(E\).
\end{proposition}

\begin{proof}
Let
\[
  R_u=\{\bar e\in q(E)\mid \bar e\sim_Qq(u)\}.
\]
This is effectively rational by Lemma~\ref{lem:VC-calculus}(a). The map \(q|_E:E\to q(E)\) is
integral affine on each of the fixed finite coset cells, and its kernel is \(C\). Hence its inverse image
of \(R_u\) is effectively rational. The height fibre
\[
 \chi^{-1}(\chi(u))\cap E
\]
is an integral-affine relation and is effectively rational. The height lemma gives the exact equality
\[
 \mathcal C_E(u)
 =(q|_E)^{-1}(R_u)\cap\chi^{-1}(\chi(u)).
\]
Indeed, the two conditions on the right say precisely that the images are conjugate in \(Q\) and the heights agree. Thus no pointwise
testing of an infinite family is required. Effective rationality and decidable non-emptiness follow from Theorem~\ref{thm:rational-calculus}.
\end{proof}

\begin{proposition}\label{prop:rigid-transition}
Let \(x\) be a finite-outer vertex of \(T_{\mathrm{blk}}\) and put
\[
 V=M_x,\qquad z=z_x,\qquad C=\langle z\rangle,\qquad
 q:V\longrightarrow Q=V/C,
\]
and let \(E,E'\leq V\) be incident edge groups. Then the two-boundary conjugacy relation \(\mathcal P_{E,E'}(V)\) is effectively rational.
\end{proposition}

\begin{proof}
Let
\[
 R_{E,E'}=
 \{(\bar e,\bar e')\in q(E)\times q(E') \mid \bar e\sim_Q\bar e'\}.
\]
Lemma~\ref{lem:VC-calculus}(b) makes this relation effectively rational. Pull it back along
\[
 q|_E\times q|_{E'}:E\times E' \longrightarrow q(E)\times q(E')
\]
using the finite coset coordinates, and intersect it with the
integral-affine subgroup
\[
    \{(e,e')\in E\times E'\mid
        \chi(e)=\chi(e')\}.
\]
By Lemma~\ref{lem:height}, the resulting set is exactly
\[
       \mathcal P_{E,E'}(V).
\]
Theorem~\ref{thm:rational-calculus} gives its effective rationality.
\end{proof}


\begin{proposition}\label{prop:rigid-two-coset}
Let \(x\) be a finite-outer vertex of \(T_{\mathrm{blk}}\) and put
\[
 V=M_x,\qquad z=z_x,\qquad C=\langle z\rangle,\qquad
 q:V\longrightarrow Q=V/C,
\]
and let \(E_1,E_2\leq V\) be incident edge groups. For every \(u,v\in V\), the two-coset problem is solvable and
\(\mathcal C_{E_1,E_2}(u,v)\) is an effectively rational relation in \(E_1\times E_2\).
\end{proposition}

\begin{proof}
Project
\[
u=e_1ve_2,\qquad e_i\in E_i,
\]
to \(Q\). Lemma~\ref{lem:VC-calculus}(c) gives an effective rational relation for the quotient solutions in
\(q(E_1)\times q(E_2)\). Pull it back along
\[
 q|_{E_1}\times q|_{E_2}:E_1\times E_2 \longrightarrow q(E_1)\times q(E_2)
\]
using the finite coset coordinates. For a pair in this pullback,
there is a unique \(k\in\Z\) such that
\[
u=e_1ve_2z^k.
\]
Applying \( \chi\) to this equation gives the following, 
\[
 k\,\chi(z) =\chi(u)-\chi(e_1)-\chi(v)-\chi(e_2).  \tag{CE}
\]
Since \(\chi(z)\neq0\), our original equation holds if and only if
\[
\chi(e_1)+\chi(e_2) = \chi(u)-\chi(v).    \tag{H}
\]
On each finite cell, (H) is an integral-affine constraint on the edge groups. Therefore, intersecting this constraint with the pulled-back quotient
relation gives exactly our entire solution set. Theorem~\ref{thm:rational-calculus} makes this intersection effectively
rational. It is clear that every solution arises in this way. 
\end{proof}

\begin{remark}
Propositions~\ref{prop:rigid-boundary},
\ref{prop:rigid-transition}, and
\ref{prop:rigid-two-coset} are the precise analogue of Pr\'eaux's reduction of the Seifert case to the quotient by the regular fibre.
For a Seifert group the central quotient is Fuchsian; for a finite-outer vertex group it is hyperbolic by
Proposition~\ref{prop:finite-outer}. Lemma~\ref{lem:VC-calculus} supplies the
output required by the global rational-intersection argument.
\end{remark}

\subsubsection*{\normalfont\bfseries Verification of hypotheses and proof of Theorem~\ref{thm:main}}

\begin{corollary}[Verification of the abstract hypotheses]
\label{cor:local-verification}
The graph of groups \(\cG_{\mathrm{blk}}\) satisfies hypotheses
\textup{(i)--(iv)} of
Proposition~\ref{prop:abstract-preaux}.
\end{corollary}

\begin{proof}
Theorem~\ref{thm:block-reduction} gives the three vertex types. Proposition~\ref{prop:elementary-package} treats the elementary
vertices, and Proposition~\ref{prop:effective-preaux} treats the pseudo-Anosov vertices. For a finite-outer vertex, membership
and conjugacy are supplied by Lemmas~\ref{lem:rigid-membership} and \ref{lem:rigid-conjugacy}, while
Propositions~\ref{prop:rigid-boundary}, \ref{prop:rigid-transition}, and \ref{prop:rigid-two-coset} give the required rational sets and
relations. Finally, Theorem~\ref{thm:block-reduction} gives four-acylindricity, so Remark~\ref{rem:acylindrical} gives hypothesis~\textup{(iv)}.
\end{proof}

\begin{proof}[Proof of Theorem~\ref{thm:main}]
If \([\phi]\) has finite order in \(\Out(G)\), apply Proposition~\ref{prop:finite-outer}. Suppose that \([\phi]\) has
infinite order. If the canonical JSJ tree is trivial, \(G\) is a closed surface group and \(G_\phi\) is a three-manifold group, so the
result follows from \cite{PreauxOrientable,PreauxNonorientable}.

Otherwise form the suspended canonical JSJ, refine each QH suspension as in Section~\ref{sec:refinement}, and perform the block
reduction of Section~\ref{sec:blocks}, obtaining the graph-of-groups
description
\[
 G_\phi\cong\pi_1(\cG_{\mathrm{blk}}).
\]

Corollary~\ref{cor:local-verification} verifies all the hypotheses of
Proposition~\ref{prop:abstract-preaux} for \(\cG_{\mathrm{blk}}\). The proposition therefore gives a
conjugacy algorithm for \(G_\phi\cong\pi_1(\cG_{\mathrm{blk}})\).

\end{proof}

\begin{remark}
	For the closed surface case we quoted \cite{PreauxOrientable, PreauxNonorientable}. However, even though \( T_{\mathrm{can}}\) is a point in this case, \( T_{\mathrm{ref}}\) will not be a point when the monodromy is reducible. In that case one can still obtain \( T_{\mathrm{blk}}\) and this will be the usual JSJ decomposition of the 3-manifold. 
\end{remark}

\end{document}